\documentclass[10pt,letterpaper]{amsart}
\usepackage{euler, epic,eepic}
\usepackage{latexsym, amssymb, amscd, amsfonts, enumerate, mathtools, fancyref}
\usepackage{xcolor}
\usepackage{amsmath}
 \numberwithin{equation}{section}
\usepackage{url, xypic, epsfig, verbatim, skull}
\input xy
\xyoption{all}

 \newlength{\baseunit}               
 \newcount{\numlines}                
\newcommand{\beq}{\begin{equation}}
\newcommand{\eeq}{\end{equation}}

\newtheorem*{tmnl}{Theorem}

\newtheorem{tm}{Theorem}
\newtheorem{pr}[tm]{Proposition}
\newtheorem{lm}[tm]{Lemma}
\newtheorem{co}[tm]{Corollary}

\theoremstyle{definition}
\newtheorem{df}[tm]{Definition}

\theoremstyle{remark}
\newtheorem{rmk}[tm]{Remark}

\newcommand{\bbC}{\mathbb{C}}

\newcommand{\bbF}{\mathbb{F}}

\newcommand{\bbP}{\mathbb{P}}
\renewcommand{\P}{\mathbb{P}}
\newcommand{\bbQ}{\mathbb{Q}}
\newcommand{\bbR}{\mathbb{R}}

\newcommand{\bbZ}{\mathbb{Z}}
\newcommand{\Z}{\mathbb{Z}}
\newcommand{\R}{\mathbb{R}}
\newcommand{\Q}{\mathbb{Q}}

\newcommand{\calB}{{ \mathcal B}}

\newcommand{\calE}{{ \mathcal E}}

\newcommand{\calI}{{ \mathcal I}}

\newcommand{\calO}{{ \mathcal O}}

\newcommand{\Hilb}{ \operatorname{Hilb} }
\newcommand{\Pic}{ \operatorname{Pic} }
\newcommand{\covol}{ \operatorname{covol} }

\newcommand{\dist}{ \operatorname{dist} }
\newcommand{\Le}{ \operatorname{Le} }

\newcommand{\Spec}{ {\operatorname{Spec}}}
\newcommand{\polyring}{S}
\newcommand{\ovS}{\overline{\polyring(2)}(\Lambda_1)}

\newcommand{\Disc}{\operatorname{Disc}}

\begin{document}
\pagestyle{plain}
\title{What is the height of 2 points in the plane?}

\address{Dept. of Mathematics, University of California, Santa Cruz, CA}
\email{jelkass@ucsc.edu}
\address{Dept. of Mathematics, University of South Carolina, Columbia, SC}
\email{thorne@math.sc.edu}

\author{Jesse Leo Kass and Frank Thorne}

\begin{abstract}
Here we describe the distribution of rational points on the Hilbert scheme of two points in the projective plane.

More specifically, we explicitly describe a two-parameter family of height functions $H_{s, t}$, such that the height function associated
to any projective embedding is equivalent to some $H_{s, t}$, up to multiplication by a bounded function. For a certain range of the parameters
$(s, t)$, we prove an asymptotic formula for the number of rational points of bounded height, and for other $(s, t)$ we obtain an upper bound. The proof establishes 
an equivalence to a lattice point counting problem,
which we solve using the geometry of numbers.

\end{abstract}
\maketitle

{\parskip=2pt 

\section{Introduction}

In this study we study the distribution of rational points on the {\itshape Hilbert scheme} $\Hilb^2(\P^2)$ of two points in the projective plane.
More generally, the {\itshape Hilbert scheme} $\Hilb^{d}(\bbP^2)$ parameterizes collections of $d$ points in $\bbP^{2}$ together with their degenerations, 
i.e., the $0$-dimensional closed subschemes of degree $d$. Since this scheme is defined over $\bbQ$, the set of rational points $\Hilb^{d}(\bbP^{2})(\bbQ)$ 
is well defined,
and in this paper we study the size of this set for $d = 2$. It is infinite, and so we fix a height function $H_i$
associated to a projective embedding $i \colon \Hilb^{2}(\bbP^2) \to \bbP^N$, 
and ask for an estimate of
\begin{equation} \label{Eqn: GeneralPointCount}
	N_{i}(B) := \# \{ x \in \Hilb^{2}(\bbP^2)(\bbQ) \colon H_{i}(x) \le B \}.
\end{equation}

The formalism of Weil's height machine yields a natural equivalence relation on height functions. 
If $i_1$ and $i_2$ are projective embeddings with $i_{1}^{*} \calO(1) = i_{2}^{*} \calO(1)$, then the ratio $H_{i_1}/H_{i_2}$ is bounded \cite[Theorem, page~22]{serre97}, and we define
any two height functions $H, H'$ with $H/H'$ bounded to be equivalent. 
Up to a bounded constant, as long as the asymptotics in \eqref{Eqn: GeneralPointCount} take a nice form (e.g., \eqref{Eqn: baryrevManin}), they
will depend only on the very ample line bundle $i^{*}\calO(1)$ and not the specific embedding $i$. 

The set of equivalence classes of height functions is parametrized by 
the Picard group $\Pic(\Hilb^{2}(\bbP^{2}))$, which in this case is free abelian of rank $2$, generated by 
the line bundles  $D_1 := i_{1}^{*} \calO(1)$ and $D_2 := i^{*}_{2} \calO(1)$ for morphisms $i_1$ and $i_2$ to be described in Section \ref{Section: HeightFunctions}.
The ample cone consists of those line bundles $(s - t)D_{1} + t D_{2}$ with $s, t > 0$, and we define
corresponding height functions 
\[
	H_{s,t}(Z) = H_{i_{1}}(Z)^{s - t} H_{i_{2}}(Z)^{t}. \label{eq:hf_form}
\]
This is in keeping with Peyre's `all the heights' philosophy \cite[Section 4]{peyre_beyond_heights}, in which he proposes studying height functions associated 
to all equivalence classes of line bundles simultaneously.

We will define these height functions $H_{i_e}$ formally in Section \ref{Section: HeightFunctions}, and we will also see that they have the following concrete interpretation.
For each integral point $Z \in \Hilb^{2}(\bbP^{2})(\bbZ)$  on the Hilbert scheme, write
$I_{Z}(e) \subseteq \Z[X_0, X_1, X_2](e)$ for the lattice of integral degree $e$ polynomials vanishing on $Z$. Defining a volume on $\bbR[X_0, X_1, X_2](e)$ by choosing the 
standard monomials for an orthonormal basis, we then have
\begin{align}
	H_{i_{e}}(Z) =& \operatorname{covol} I_{Z}(e). \label{Eqn: IntroHeightFunctions} 
\end{align}
The height function $H_{s, t}$ extends to $\Pic(\Hilb^{2}(\bbP^{2})) \otimes \bbR$; that is, it is well-defined for arbitrary real numbers $s$ and $t$. If $s$ and $t$ are {\itshape positive},
then there will be only finitely many $\bbZ$-points of bounded height, so that we may associate a counting function
$N_{s, t}$ as in \eqref{Eqn: GeneralPointCount} to this height function. We prove the following:

\begin{tmnl}\label{thm:main}
For $s, t > 0$ we have
	\begin{equation} \label{Eqn: MainThmCount}
		N_{s, t}(B) = c_{s, t} B^{3/t} + O\left(B^{\frac{2}{t}} + B^{\frac{3}{s}} (\log B)\right),
	\end{equation}
where 
\begin{equation}\label{eqn:mainterm0}
c_i = \frac{\pi}{3 \zeta(3)} \sum_{a, b, c} \frac{ (a^2 + b^2 + c^2)^{\frac{3}{2} - \frac{3}{2} \cdot \frac{s}{t}} }
{a^6 + 2b^2a^4 + 2c^2a^4 + 2b^4a^2 + 5c^2b^2a^2 + 2c^4a^2 + b^6 + 2c^2b^4 + 2c^4b^2 + c^6}.
\end{equation}
\end{tmnl}

This is an asymptotic estimate for $\frac{s}{t} > 1$ and an upper bound for $s/t \leq 1$. This complements work of Le Rudulier \cite{lerudulier}, who treated the
case $s = 0$ after removing a thin set containing infinitely many points of bounded height, as we will describe shortly.

As discussed earlier, the height function associated to any projective embedding
$i \colon \Hilb^{2}(\bbP^{2}) \to \bbP^{N}$ will differ from some $H_{s,t}$ by a bounded function. We therefore immediately obtain:
\begin{co}\label{Co: MainCor}
Let $i$  be a projective embedding of  $\Hilb^2(\bbP^2)$   with $i^* \calO(1)$ equivalent to $(s -t) D_1 + t D_2$ in
$\Pic( \Hilb^2(\bbP^2))$ with  $\frac{s}{t} > 1$. Then we have
\begin{equation}\label{eqn:maintermCorForm}
	N_{i}(B) \asymp B^{3/t}.
\end{equation}
\end{co}

\smallskip
Recall that the {\itshape Batyrev-Manin conjecture} \cite{batyrev90}
predicts that, after possibly removing a thin set from $\Hilb^{2}(\bbP^{2})(\bbQ)$ and passing to a finite extension, we have
 $N_{s,t}(B)  \asymp  B^\alpha \log(B)^\beta$ for $\alpha$ and $\beta$ constants depending on the birational geometry of $\Hilb^{2}(\bbP^{2})$ (specifically the structure of the effective cone). 
In Section~\ref{Section: BatyrevManin} we compute these constants and confirm that they agree with our results.
 
 \medskip
 {\itshape Summary of the proof.} The idea of our proof can be summarized succinctly. Consider the map sending a degree $2$ closed subscheme $Z$ to the line it spans.  This is a map $\Hilb^{2}(\bbP^{2}) \to (\bbP^{2})^{\vee}$ to the dual projective space with fibers that are projective spaces.  We count rational points on $\Hilb^{2}(\bbP^{2})$ by estimating the points on a given fiber and then summing over all possible fibers.  Indeed, it will be seen in the proof that each fiber contains a positive proportion of the points. 
 
More specifically, we establish a correspondence (Lemma \ref{Lemma: DescribeQPoints}) between rational (equivalently, integral) points on 
$\Hilb^{2}(\bbP^{2})$ and appropriate $\mathbb{Q}$-vector spaces $(W_1, W_2)$; if $Z \in \Hilb^{2}(\bbP^{2})(\mathbb{Q})$ describes
a pair of points $Z_1, Z_2$ in $\bbP^2(\mathbb{Q})$, then $W_1$ and $W_2$ are the spaces of linear and quadratic forms vanishing on $Z_1$ and $Z_2$ respectively.
We then give this correspondence an integral structure (Lemma \ref{Lemma: Bijection}), allowing us to consider lattices $\Lambda_i$ in place of the  vector spaces $W_i$.

Sections \ref{Section: HeightFunctions} and \ref{sec:weil} describe the ample cone of $\Hilb^{2}(\bbP^{2})$ and define the height functions $H_{s, t}$. 
With this machinery developed, we can then refomulate our theorem as a lattice point counting problem
(see Proposition \ref{pr:geocount}), for which no further algebraic geometry is required. To address this problem we introduce some tools
from the geometry of numbers in 
Section \ref{sec:gon}, building upon work of Schmidt \cite{schmidt95}, and we then prove our main theorem in Section \ref{sec:proof}.

\medskip
{\itshape Comparison to work of Le Rudulier.}
The main results of this paper should be compared to results of Le Rudulier \cite{lerudulier}, 
who estimates the rational points on $\Hilb^2(\bbP^2)$ with respect to a height function
associated to a metrization of the anti-canonical line bundle $-K =  3 D_2- 3 D_{1}$. This is not covered by Corollary~\ref{Co: MainCor} as $s/t=0$, and indeed the anti-canonical bundle is not very ample. Instead, 
it is the pullback of an ample line bundle on the symmetric product $\operatorname{Sym}^{2}(\bbP^2)$ under the Hilbert--Chow morphism $\Hilb^{2}(\bbP^{2}) \to \operatorname{Sym}^{2}(\bbP^{2})$, which
contracts the locus $E$ of nonreduced subschemes. The set $\{ x \in \Hilb^{2}(\bbP^{2})(\bbQ) \colon H(x) \le B \}$ is infinite for $B$ is sufficiently large, but if $E(\bbQ)$ is removed, this set becomes finite, and Le Rudulier proves
\begin{equation} \label{Eqn: LeRudulier}
	\# \{ x \in \Hilb^{2}(\bbP^2)(\bbQ)  - E(\bbQ) \colon H_{i}(x) \le B \} \thicksim \frac{2(24 + \pi^{2})}{3 \zeta(3)^{2}}B \log(B) \text{ as $B \to \infty$ \cite[Th\'{e}or\`{e}me~4.2]{lerudulier}.}
\end{equation}
(Le Rudulier also estimates the size of the set obtained by additionally removing the thin set consisting of pairs $\{ p, q\}$ with $p, q \in \bbP^{2}(\bbQ)$; this set  arises when considering Peyre's conjecture.
To obtain \eqref{Eqn: LeRudulier}, add the two estimates in \cite[Th\'{e}or\`{e}me~4.2]{lerudulier}.)

Le Rudulier's proof, in contrast to ours, does not use the fiber bundle structure of $\Hilb^{2}(\bbP^{2})$. (Indeed, in her setup, the number of points on a fixed fiber of $\Hilb^{2}(\bbP^{2}) - E \to (\bbP^{2})^{\vee}$ of height at most $B$ is $\asymp B$, so that no single fiber contributes a positive proportion of the points.) Instead, she breaks up the set $\Hilb^{2}(\bbP^{2})(\bbQ)$ into two subsets: the subset consisting of pairs
$\{ p, q \}$ of $\bbQ$-points and the subset of Galois conjugate pairs of points individually defined over some quadratic extension. For the second subset she estimates the points using  work of Schmidt \cite{schmidt95}; counting  the first subset is equivalent to counting  the elements of $\bbP^{2}(\bbQ) \times \bbP^{2}(\bbQ)$ with respect to the height function $H([A,B,C], [D, E, F]) = \sqrt{A^2+B^2+C^2} \sqrt{D^2+E^2+F^2}$, and this counting problem can be analyzed directly. 

In Section \ref{Section: LeRHeight} we introduce Le Rudulier's height function. (See Definition \ref{Def: LeRudHeight}.) Then, building on this, in Section \ref{Section: QuadAlgebra} we further study the connection
between integral points of $\Hilb^{2}(\bbP^{2})$ and quadratic rings. To each integral point $Z \in \Hilb^{2}(\bbP^{2})$ we may associate the
quadratic ring $A = H^0(Z, \mathcal{O}_Z)$. We obtain an explicit formula for $\Disc(A)$, and then in Proposition~\ref{Prop: DiscBound} 
we prove that $|\Disc(A)| \ll H_{-2, 2}(Z)$. This is done by directly relating $|\Disc(A)|$ to the explicit form of 
Le Rudulier's height function, and then relying on the relationship of this height function to the anti-canonical line bundle.

We conclude by mentioning a couple of additional related works. First, there is 
work of M\^{a}nz\u{a}\c{t}eanu \cite{manzateanu}, obtaining an analogue of Le Rudulier's results over function fields. 
There is also work of Sawin \cite{sawin}, showing that removing a thin set is indeed necessary, in the more general case of $\Hilb^2(\P^n)$;
in particular, a `freeness' condition of Peyre \cite{peyre_freeness} is not enough to characterize the points which must be removed.

\section{Background on the Hilbert scheme} \label{Section: BackgroundHilb}
In this section we recall the definition of the Hilbert scheme $\Hilb^2(\bbP^2)$, and establish a bijection (Lemma \ref{Lemma: Bijection}) 
between its integral points and certain pairs of lattices. We then exhibit an  infinite family of embeddings of $\Hilb^2(\bbP^2)$ into projective space,
and apply a result of Schmidt \cite{schmidt67} to describe the associated height functions in terms of the covolumes of our lattices. 

Finally, we describe the ample cone $\Hilb^2(\bbP^2)$, together with an associated two-parameter family of height functions $H_{s, t}$ defined
in terms of lattice counting.    By the formalism of Weil's height machine, the height function associated to any projective embedding of $\Hilb^2(\bbP^2)$ will be equivalent
to one such $H_{s, t}$. 

\medskip
\subsection{Definitions; parametrization of rational points}
The {\itshape Hilbert scheme}, denoted $\Hilb^{2}(\bbP^{2})$ or $\Hilb^2$, is a projective scheme that parameterizes equivalently pairs of points in $\bbP^{2}$ and their degenerate limits, i.e.~degree $2$ closed subschemes.  More formally, by Yoneda's lemma, the Hilbert scheme is uniquely characterized by  its sets of $R$-points $\Hilb^{2}(R)$ as $R$ varies over all rings.  The Hilbert scheme is defined by setting $\Hilb^{2}(R)$ equal to the set of $R$-flat closed subschemes $Z \subset \bbP^{2}_{R}$ with the property that, for every $s \in \Spec(R)$, the fiber $Z_s$ has dimension $0$ and degree $2$.  The elements of $\Hilb^{2}(\bbZ)$ can also be described in terms of quadratic algebras; see Section~\ref{Section: QuadAlgebra}.

The Hilbert scheme  $\Hilb^2$ exists as a projective scheme that is $\bbZ$-smooth of relative dimension $4$.  Indeed, over a field this result is \cite[Theorem~2.4]{fogarty68}, and over $\bbZ$, existence as a projective scheme is a very general theorem of Grothendieck \cite[Theorem~3.2]{fga}, and smoothness over $\bbZ$ follows from \cite[Corollary~8.10]{hartshorne}. We will be  primarily interested in the sets  $\Hilb^{2}(\bbZ)$ and $\Hilb^{2}(\bbQ)$ of  integral and rational points respectively, and these sets have more concrete descriptions we now give.

If we consider elements of $\Hilb^{2}(\bbQ)$ as being degree $2$ closed subschemes of $\bbP^{2}_{\bbQ}$, then every such element can be described as the solution set to a system of equations consisting of a linear polynomial and a quadratic polynomial, i.e.~a system consisting of the equations
\begin{equation} \label{Eqn: EquationsForSubscheme}
	\ell := a_0 X_0 + a_1 X_1 + a_2 X_2 \text{ and } q := \sum_{i+j+k=2} b_{i,j,k} X_{0}^i X_1^j X_2^k.
\end{equation}
These equations define a degree $2$ closed subscheme of $\bbP^{2}_{\bbQ}$ provided the quadratic equation is not a multiple of the linear equation and the linear equation is nonzero.  Indeed, the subscheme is
\[
	Z := \operatorname{Proj} \frac{\bbQ[X_0, X_1, X_2]}{\ell(X_0, X_1, X_2), q(X_0, X_1, X_2)}.
\]

The following definition and lemma make the correspondence between the polynomials in \eqref{Eqn: EquationsForSubscheme} and elements of $\Hilb^{2}(\bbQ)$ precise. 
\begin{df}
	Given a ring $R$ and an integer $d$, let $S_{R}(d)$ equal the $R$-module of homogeneous degree $d$ polynomials in $X_0, X_1, X_2$.  
	
	Given an $R$-valued point  $Z \in \Hilb^{2}(R)$ of the Hilbert scheme, define $I_{Z}(d) \subset S_{R}(d)$ to be the $R$-module of homogeneous degree $d$ polynomials vanishing on $Z$ and define $I_{Z} := \oplus_{d=0}^{\infty} I_{Z}(d)$ to be the homogeneous ideal of polynomials vanishing on $Z$. 
\end{df}

We state the lemma below for an arbitrary field $k$ rather than just $\bbQ$ because we need the more general statement for Lemma~\ref{Lemma: QuotientIsLattice}.
\begin{lm} \label{Lemma: DescribeQPoints}
	Let $k$ be a field.  Then the rule 
	\begin{equation} \label{Eqn: DescribePoints}
		Z \mapsto (I_Z(1), I_Z(2))
	\end{equation}
	defines a bijection between the set $\Hilb^{2}(k)$ of $k$-valued points and the set of pairs of $k$-subspaces $(W_1, W_2)$ with $W_1 \subset S_{k}(1)$ of dimension $1$ and $W_2 \subset S_{k}(2)$ of dimension $4$ and containing $S_{k}(1) \cdot W_1$.  
	
	The inverse map is defined by 
	\begin{equation} \label{Eqn: InverseMap}
		(W_1, W_2) \mapsto \operatorname{Proj} k[X_0, X_1, X_2]/I_{W}
	\end{equation}
		 for $I_{W} \subset S_{k}$ the ideal generated by $W_1$ and $W_2$.  
\end{lm}
\begin{proof}
	Recall that the Proj construction defines a bijection between closed subschemes of $\bbP^{2}_{k}$ and saturated homogeneous ideals of $S_{k}$, and under this correspondence, the subschemes of dimension 0 and degree 2 correspond to the ideals $I$ with the property that $I(d)$ has codimension $2$ in $S_{k}(d)$ for all large $d$.  
The proof will show that these subschemes in fact correspond to ideals with this property for all $d \geq 1$.
	
	To show that \eqref{Eqn: InverseMap} is well-defined, and hence (by the bijection described above) injective, 
	write $\ell$ for a generator of $W_1$ and $q$ for a polynomial such that $W_2 = \langle X_0 \ell, X_1 \ell, X_2 \ell, q \rangle$. We then see that
	\begin{align}
		\dim I_{W}(d) =&
		\dim \ell \cdot S_{k}(d-1) + \dim q \cdot S_{k}(d-2) - \dim \ell \cdot q \cdot S_{k}(d-3)  \label{Eqn: IdealDimCount} \\
\notag		=& \binom{d+1}{2} + \binom{d}{2} - \binom{d-1}{2} \text{ for $d \ge 1$}\\
			=&\dim S_{k}(d) -  2, \notag
	\end{align}
as needed.
	
	To show that \eqref{Eqn: InverseMap}  is surjective,  given a degree $2$ closed subscheme, the exact sequence 
	\[
		0 \to I_{Z}(d) \to S_{k}(d) \to H^{0}(Z, \calO_{Z}(d))
	\]
	shows that $\dim_{k} I_{Z}(1) \ge 1$ and $\dim_{k} I_{Z}(2) \ge 4$.  Thus we can pick a $1$-dimensional subspace $W_1 \subset I_{Z}(1)$ and a $4$-dimensional subspace $W_2 \subset I_{Z}(2)$ that contains $S_{k}(1) \cdot W_1$.  From what we've already proven, the homogeneous ideal generated by $W_1+W_2$ defines a  degree $2$ closed subscheme $Z_0$.  By construction $Z_0 \supset Z$, but this inclusion must be an equality since the subschemes have the same dimension and degree. 
	
\end{proof}

We now turn our attention to the integral points $\Hilb^{2}(\bbZ)$.  As with rational points, every closed subscheme is defined by a system of equations of the form \eqref{Eqn: EquationsForSubscheme}, but the constraints on the equations are different.  For example, the equations  $X_0$ and $2 X_1^2$ do not define an element of $\Hilb^{2}(\bbZ)$ even
 though the quadratic equation is not a multiple of the linear equation.  Indeed, if these equations defined such an element, then the mod $2$ reduction would be a degree $2$ closed subscheme of $\bbP^{2}_{\bbF_{2}}$, but this reduction is the line $\{ X_0 =0\} \subset \bbP^{2}_{\bbF_2}$. Observe that $I_{Z}(2)$ is not a primitive sublattice of $S_{\bbZ}(2)$, and we recover  bijectivity by imposing primitivity.  (Recall a sublattice $\Lambda \subset \bbZ^{n}$ is said to be primitive if  $v \in \mathbb{Z}^{\oplus n}$ and $n \in \mathbb{Z}-\{0\}$ and satisfy $n \cdot v \in \Lambda$ then $v \in \Lambda$.)

\begin{lm}	\label{Lemma: Bijection}
	The rule \eqref{Eqn: DescribePoints} defines a bijection between  $\Hilb^{2}(\bbZ)$ and the set of pairs of primitive lattices $(\Lambda_1, \Lambda_2)$ with $\Lambda_1 \subset S_{\bbZ}(1)$ of dimension $1$ and $\Lambda_2 \subset S_{\bbZ}(2)$ of dimension $4$ and containing $S_{\bbZ}(1) \cdot \Lambda_1$.  
\end{lm}
\begin{proof}
	We  deduce the result by arguing that everything is determined by what happens over $\bbQ$ and then citing Lemma~\ref{Lemma: DescribeQPoints}.  Consider the commutative diagram
	\[
\begin{CD}
\Hilb^{2}(\bbZ)     @>>>  \{ \text{suitable sublattices ($\Lambda_1, \Lambda_2)$} \}\\
@VVV        @VVV\\
\Hilb^{2}(\bbQ)     @>>>  \{ \text{suitable subspaces ($W_1, W_2)$} \}.\\
\end{CD}
	\]
	Here the horizontal maps are the maps $Z \mapsto (I_{Z}(1), I_{Z}(2))$, the left-hand vertical map is the tautological injection, and the right-hand vertical map is given by extending scalars to $\bbQ$.  The top-most horizontal map is well-defined because $S_{\bbZ}(i)/\Lambda_i$ injects into $H^{0}(Z, \calO_{Z}(i))$ which is torsion-free since $Z$ is $\bbZ$-flat by definition.  
	
	All maps except for the top map are immediately seen to be bijective.  Indeed, the  bottom map is bijective by Lemma~\ref{Lemma: DescribeQPoints},  the left-most map is bijective because $\Hilb^2$ is proper, and the right-most map is bijective by a direct argument:  given $(W_1, W_2)$, the unique primitive sublattices $\Lambda_1 \subset W_1$, $\Lambda_2 \subset W_2$ define the unique pair $(\Lambda_1, \Lambda_2)$ mapping to $(W_1, W_2)$.  We conclude that the top map is bijective as well. 
\end{proof}

\subsection{A parameterized family of height functions} \label{Section: HeightFunctions}
We now construct a family of maps $i_e$ ($e = 1, 2, \dots$) from $\Hilb^{2}$ to projective space, to which we can associate height functions $H_e$ which are easily 
described. We construct $\Hilb^{2} \to \bbP^{N}$ by constructing a map of sets  $\Hilb^{2}(R) \to \bbP^{N}(R)$ for every ring $R$ in a manner that is functorial in $R$.  Yoneda's lemma implies that such maps define a morphism of schemes,
as required to invoke Weil's height machine.

The maps are constructed by first mapping $\Hilb^{2}$ to a Grassmannian scheme and then taking a standard projective embedding of the Grassmnannian, namely the Pl\"{u}cker embedding.  Recall that a Grassmannnian scheme parameterizes linear subspaces of a fixed vector space.  In the present context, we are interested in  the Grassmannian parameterizing codimension 2 subspaces of $S(e)$.  Observe that if $Z \in \Hilb^{2}(R)$ is a point, then the $R$-module $I_{Z}(e)$ of degree $e$ homogeneous polynomials vanishing on $Z$ is a submodule of the $R$-module $S_{R}(e)$ of homogeneous degree $e$ polynomials with coefficients in $R$.  The following lemma shows that this submodule defines an $R$-point of the Grassmannian  parameterizing corank $e$ subspaces of $\bbZ[X_0, X_1, X_2](e)$.  

\begin{lm} \label{Lemma: QuotientIsLattice}
	If $R$ is a ring and $Z \subset \bbP^{2}_{R}$ is a $R$-flat subscheme with fibers of dimension $0$, degree $2$, then for $e=1, 2, \dots$, the quotient of $S_{R}(e)$ by $I_{Z}(e)$ is locally free of rank $2$.
\end{lm}
\begin{proof}
	We prove the lemma by showing that the quotient module is  isomorphic to $H^{0}(Z, \calO_{Z})$ which is locally free of rank $2$ by definition.  To see that these modules are isomorphic, observe that the restriction map $H^{0}( \bbP^{2}_{R}, \calO_{\bbP^2_{R}}(e)) \to H^{0}( \bbP^{2}_{R}, \calO_{Z}(e))$ is one of the maps in the long exact sequence induced by 
	\[
		0 \to \calI_{Z}(e) \to \calO_{\bbP^2_{R}}(e) \to \calO_{Z}(e) \to 0.
	\]
	From the long exact sequence, we get that the cokernel of  $H^{0}( \bbP^{2}_{R}, \calO_{\bbP^2_{R}}(e)) \to H^{0}( \bbP^{2}_{R}, \calO_{Z}(e))$  is contained in $H^{1}( \bbP^{2}_{R}, I_{Z}(e))$, but this last cohomology group is zero.  Indeed, by the theory of cohomology and base change, it is enough to show that the group vanishes when $R=k$ is a field.  In this case, Lemma~\ref{Lemma: DescribeQPoints} shows that $\calI_{Z}$ is the complete intersection of a line and a quadratic, so it has a resolution of the form
	\[
		0 \to \calO(-3) \to \calO(-2) \oplus  \calO(-1) \to \calI_{Z} \to 0.
	\]
	Vanishing follows by passing to the long exact sequence and using the standard computation of $H^{i}(\bbP_k, \calO(d))$.

	We now complete the argument by observing that $\calO_{Z}$ is isomorphic to $\calO_{Z}(e)$ because $Z$ is supported on a finite set.
\end{proof}

\begin{df}\label{def:ie}
Let $G$ be the Grassmannian parameterizing rank $2$, locally free quotients of $S_{\bbZ}(e)=\bbZ[X_0, X_1, X_2](e)$.  Define $\widetilde{i}_{e} \colon \Hilb^{2} \to G$ to be the unique morphism with the property that, for a ring $R$, the map $\Hilb^{2}(R) \to G(R)$ is the map sending $Z$ to $I_{Z}(e) \subset S_{R}(e)$. This is well-defined by Lemma~\ref{Lemma: QuotientIsLattice}.

Let  $V_{e} := \bigwedge^{ \binom{e+2}{2}-2} S_{\bbZ}(e)$ be the  exterior power of the module of degree $e$ polynomials in $X_0, X_1, X_2$ and $\bbP V_{e}$ be the associated projective space parameterizing $1$-dimensional subspaces of $V_{e}$.  Define $i_{e} \colon \Hilb^{2} \to \bbP V_{e}$ to be the composition of $\widetilde{i}_{e}$ with the Pl\"{u}cker embedding of $G$, so for a ring $R$, $\Hilb^{2}(R) \to 
\bbP V_{e}(R)$ is the map $Z \mapsto \bigwedge^{\text{top}} I_{Z}(e)$. 
\end{df}
One can show that $i_e$ is a projective embedding for $e \ge 2$. We only need this result over $\overline{\bbQ}$, and this is  \cite[Lemma~3.8]{li03}.  The morphism $i_1$ is not a projective embedding as the fibers are projective spaces $\bbP^2$   \cite[Proposition~3.12]{li03}.

We are primarily interested in using the morphisms $i_{e}$ to describe the height functions on $\Hilb^2$.  We recall the basic definitions from the introduction.
\begin{df}
We call the map
	\begin{align}
	H_{\text{Euc}} \colon  \bbP^{N}(\bbQ) & \mapsto \bbR^+ \nonumber \\
	 x = [x_0, \dots, x_N] & \mapsto \sqrt{x_0^2 + x_1^2 + \dots + x_N^2} \label{eq:height1},
	 \end{align} 
	 the \textbf{Euclidean height function} on $\bbP^{N}$,
	 where $(x_0, \dots, x_N) \in \bbZ^{N}$ is a primitive vector that represents $x$.  
	 
	Given a morphism $f \colon V \to \bbP^{N}_{\bbQ}$ of a variety $V/\bbQ$ to projective space, the \textbf{associated Euclidean height function} is
	$H_f := H \circ f$.
	 
	 The Euclidean height function on $\bbP V_{e}$ is defined by identifying this projective space with $\bbP^{N}$ using the monomial basis of $V_{e}$ (i.e.~the basis $x_0^e, x_0^{e-1} x_{1}, x_{0}^{e-1} x_2,...$).  
\end{df}

\begin{rmk}
	We could alternatively defined the Euclidean height in terms of a vector $(x_0, \dots, x_N)$ that is not necessarily primitive.  If $(x_0, \dots, x_n)$ is an arbitrary nonzero vector, then the height of the point $x$ it represents is
	\begin{equation} \label{Eqn: AltEucHt}
		H_{\text{Euc}}(x) = \frac{\sqrt{x_0^2 + x_1^2 + \dots + x_N^2}}{I(x_0, \dots, x_n)},
	\end{equation}
	where $I(x_0, \dots, x_n) \subset \bbZ$ the ideal generated by the coordinates $x_0, \dots, x_n$.  This idea will appear in our
	discussion of Le Rudulier's height function (see Definition \ref{Def: LeRudHeight}).
\end{rmk}

\begin{rmk}
	The height function $\max_i |x_i|$ on $\bbP^N$ is often seen in place of \eqref{eq:height1}. The quotient of these height functions is bounded above and below by a bounded function depending only on $N$, making them equivalent in the context of Weil's height machine (see Section \ref{sec:weil}).
\end{rmk}

When $f$ is the Pl\"{u}cker embedding of a Grassmannian, Schmidt described the associated Euclidean height function $H_{f}$ in \cite[Theorem~1]{schmidt67}.  By projectivity, we can represent a given rational point $x \in G(\bbQ)$ by a primitive sublattice $\Lambda \subset S_{\bbZ}(e)$ (i.e.~we can represent $x$ by an integral point).  Schmidt shows that 
\begin{equation}\label{eq:schmidt_height_Grass}
	H_{f}(x) = \text{covolume of $\Lambda$.}
\end{equation}
(Over $\Q$,
this is easily proved; the real content of Schmidt's result is a generalization to number fields.)
Recall the covolume of a lattice is defined to be the volume of a fundamental parallelotope, which can be expressed algebraically as 
\begin{equation}\label{eq:covol}
	\text{covolume of }\Lambda = \det( \langle w_i, w_j \rangle)^{1/2}
\end{equation}
for $w_1, \dots, w_k$ a basis for $\Lambda$ and $\langle \cdot, \cdot \rangle$  the standard inner product (which makes the monomial basis into an orthonormal basis).

Observe that the submodule $I_{Z}(e) \subset S_{\bbZ}(e)$ defined by $Z \in \Hilb^{2}(\bbZ)$ is a primitive sublattice since the quotient is torsion-free by Lemma~\ref{Lemma: QuotientIsLattice}. 
We therefore conclude from \eqref{eq:schmidt_height_Grass} that the height function associated to the morphism $i_e \ : \Hilb^2 \rightarrow \bbP V_e$ is
\begin{equation}\label{eq:schmidt_height}
	H_{e}(x) := H_{i_e}(x) = \text{covolume of $I_{Z}(e) \subset S_{\bbZ}(e)$}
\end{equation}
Up to an ineffective constant, all of the height functions $H_e$ are determined by $H_1$ and $H_2$, as we now explain.

\subsection{Weil's height machine; height functions associated to the ample cone}\label{sec:weil}

Part of {\itshape Weil's height machine} \cite[Theorem~B.3.2]{hindry00} asserts the following.
Given two embeddings $i$ and $i'$ of a variety $V/\Q$ into projective spaces, such that
$i_{1}^{*}\calO(1)$ is isomorphic to $i_{2}^{*} \calO(1)$, the ratio of the associated height functions $H_{i}$ and $H_{i'}$ is bounded above and below.  

This associates a well-defined equivalence class of height functions to any very ample line bundle (i.e.~any line bundle of the form $i^{*} \calO(1)$).  Weil's height machine then 
extends this to arbitrary line bundles.  The main result is that there is a unique function
\begin{equation} \label{Eqn: HeightMachine}
	\Pic( \Hilb^{2}) \otimes \bbR \to \frac{\{\text{functions $\Hilb^{2}(\bbQ) \to \bbR$}\}}{\{\text{bounded functions}\}}
\end{equation}
that transforms tensor product into multiplication, satisfies a functoriality property we omit, and is normalized to agree with the definition already given for very ample line bundles.  

We now introduce our two-parameter family of height functions:
\begin{df}
	
	For $s, t \in \bbR$ define
	\begin{equation} \label{eq:height_def}
			H_{s,t}(Z) = (\text{covolume of $I(1)$})^{s-t} \cdot (\text{covolume of $I(2)$})^{t},
	\end{equation}
	and for $B \in \bbR^+$ define
	\begin{equation} \label{eq:counting_def}
	 N_{s, t}(B) := \# \{ x \in \Hilb^{2}(\bbP^2)(\bbQ) \colon H_{s, t}(x) \le B \}.
	 \end{equation}
\end{df}
Writing $D_e \in \Pic(\Hilb^2) := i_{e}^{*} \calO(1)$, the height function $H_{s,t}$ represents the image of $D_1^{\otimes s-t} \otimes D_2^{\otimes t}$ under \eqref{Eqn: HeightMachine}, and abusing language, we will say that it is the height function associated to this line bundle.

The definition of $H_{s, t}$ is motivated by the structure of the Picard group of $\Hilb^2$.  The vector space $\Pic(\Hilb^2) \otimes \bbR$ is generated by $D_1$ and $D_2$ by the main result of \cite{fogarty73}.  In terms of these generators, the nef cone, i.e.~the closure of the cone spanned by ample divisors, equals the cone of all   $D_1^{\otimes s-t} \otimes D_2^{\otimes t}$ with $s, t \ge 0$ by \cite[Theorem~3.14]{li03}.  In other words, the height functions $H_{s,t}$ with $s, t \ge 0$ are exactly the height functions associated to ample line bundles and their limits --- and hence we take them as our principal object of study.

In the remainder of the paper, we will focus on estimating the  counting functions $N_{s, t}(B)$ associated to the height functions $H_{s, t}$. 
This is a geometry of numbers question, which we will formally restate in Proposition \ref{pr:geocount} and then address using analytic number theory.
No more algebraic geometry will be required in the proof of Theorem \ref{thm:main}.

  We conclude this section by remarking on the relation between $H_{s,t}$ and some height functions that appear in the literature.

\begin{rmk}
	The line bundle $D_{e}$ is isomorphic to $D_{1}^{\otimes 2-e} \otimes D_{2}^{\otimes e-1}$ by \cite[Proposition~3.1(1)]{arcara13}, so $H_{e}$ agrees with $H_{s,t}$ for $(s, t)=(1, e-1)$  up to multiplication by a bounded function. The formalism of Weil's height function does not provide explicit information about this bounded function, but it should be possible to study it by other means.
	For example, numerical computations suggest that for $e=3$ we have
	\[
		.68 \cdot \frac{\operatorname{covol} I_{Z}(2)^2}{\operatorname{covol} I_{Z}(1)} \le \operatorname{covol} I_{Z}(3) \le  \frac{\operatorname{covol} I_{Z}(2)^2}{\operatorname{covol} I_{Z}(1)}.
	\]
\end{rmk}

\subsection{Le Rudulier's height function} \label{Section: LeRHeight}
Here we describe the height function used by Le Rudulier in \cite{lerudulier} and its relation to $H_{s, t}$.  The relationship will be used in Section~\ref{Section: QuadAlgebra} to bound the discriminant of a quadratic algebra by a height.  

Le Rudulier focuses on a height function associated to the anticanonical divisor $-K$.  The anticanonical divisor is not ample, but it is the pullback of an ample divisor under a map to a projective variety.  Recall that sending a closed subscheme $Z$ to its support defines a morphism $\Hilb^{2} \to \operatorname{Sym}^{2}(\bbP^{2})$ to the symmetric square.  Under this morphism, the anticanonical divisor of the symmetric square pulls back to $-K$.

In terms of the generators $D_1$ and $D_2$ that we are using, we have
\begin{equation}		\label{Eqn: CanonicalDivisor}
	K = D_1^{\otimes 3} \otimes D_{2}^{\otimes -3}.
\end{equation}
This is a consequence of \cite[Proposition~3.1(1)]{arcara13}  together with the the proof of \cite[Theorem~2.5]{arcara13} (the proof shows that $K$ equals $H^{-\otimes 3}$ for a divisor $H$ that is computed to be $H = D_1^{-1} \otimes D_{2}$ in the proposition).

We have associated to $-K$ the height function $H_{0,3}$.  This is different from the height function Le Rudulier works with.  She works with a height function constructed from the symmetric square.  The height function is constructed somewhat generally in \cite[Proposition~1.33]{lerudulier}.  The product of the Euclidean heights defines a height function on $\bbP^{2} \times \bbP^{2}$. This height function is invariant under the involution $(p, q) \mapsto (q, p)$ and thus it induces a height function on the symmetric square $\operatorname{Sym}^{2}(\bbP^{2})$.  The pullback of this last height function under the Hilbert--Chow morphism $\Hilb^{2} \to \operatorname{Sym}^{2}(\bbP^{2})$ is the height function that Le Rudulier works with.  We now give a more explicit description of Le Rudulier's height function that will be used later in bounding the discriminant.

Recall the ideal of a given $[Z] \in \Hilb^{2}(\bbZ)$ is generated by a linear polynomial $\ell \in S(1)$ and a quadratic polynomial $q \in S(2)$.  Le Rudulier's height is defined in terms of the solutions to
\begin{equation}	\label{Eqn: SysOfEqn}
	\ell(x, y, z) = q(x, y, z) =0.
\end{equation}

We record the following fact as we will use it multiple times.
\begin{lm}
	Up scalar multiplication, there are at most 2 nontrivial solutions \eqref{Eqn: SysOfEqn} with $(x, y, z) \in \bbQ^{\oplus 3}$.  There is always a solution over some quadratic extension of $\bbQ$.
\end{lm}
\begin{proof}
	Find a linear parameterization of the solutions to $\ell(x, y, z)=0$ and then plug it into $q$ to reduce to the analogous claim for a homogeneous polynomial in $2$ variables.  Then the claim reduces to the quadratic formula.
\end{proof}

\begin{df}\label{Def: NonredSplitNonsplit}
	Let $[Z] \in \Hilb^{2}(\bbZ)$.  We say that $Z$ is \textbf{nonreduced} if \eqref{Eqn: SysOfEqn} has exactly 1 nontrivial rational solution up to scaling.  When the system has exactly 2 nontrivial solutions up to scaling, we say that $Z$ is \textbf{split}.  Otherwise we say that $Z$ is \textbf{nonsplit}.
\end{df}

\begin{rmk}
	We will prove in Section~\ref{Section: QuadAlgebra} that the terms ``nonreduced", ``split", and ``nonsplit" coincide with their use in algebra.  In other words, we will show that, when $Z$ is nonreduced, the algebra $A:=H^{0}(Z, \calO_{Z})$ contains nilpotent elements.  Similarly, for $Z$ split, the algebra $A$ is a split algebra in the sense that $A \otimes_{\bbZ} \bbQ$ is isomorphic to $\bbQ \times \bbQ$.  When $Z$ is nonsplit, $A \otimes_{\bbZ} \bbQ$ is a quadratic field that contains $A$ as a (possibly nonmaximal) order. 
\end{rmk}

We now give  the definition of the height function. The following is (up to equivalence) the height function associated to the 
anticanonical bundle $-K$, as
described by Le Rudulier in 
\cite[Proposition 1.33 and Section 3.2]{lerudulier}.

\begin{df}  \label{Def: LeRudHeight}
	Given a ring $R$ and a vector $v \in R^{\oplus n}$, write $I(v) \subset R$ for the ideal generated by the components of $v$.  For $v \in \bbQ^{n + 1}$ a nonzero vector, write $[v] \in \bbP^{n}(\bbQ)$ for the associated rational point of projective space.

	The height function $H_{\Le} \colon \Hilb^{2}(\bbZ) \to \bbR$ is as follows.
	
	If $[Z] \in \Hilb^{2}(\bbZ)$ is nonreduced, let $v \in \bbZ^{\oplus 3}$ be a nontrivial integral solution to \eqref{Eqn: SysOfEqn}. Set
	\begin{align*}
		H_{\Le}([Z]) =& H_{\text{Euc}}( [v] ) \cdot H_{\text{Euc}}( [v] ) \\
			=&  \left( \frac{||v||}{\operatorname{Norm}(I(v))} \right)^2.
	\end{align*}
	
	If $[Z] \in \Hilb^{2}(\bbZ)$ is split, let $v, w \in \bbZ^{\oplus 3}$ be linearly independent integral solutions to \eqref{Eqn: SysOfEqn}.  Set
	\begin{align*}
		H_{\Le}([Z]) =&	H_{\text{Euc}}([v]) \cdot H_{\text{Euc}}( [w])	\\
			=&  \frac{||v||}{\operatorname{Norm}(I(v))} \cdot \frac{||w||}{\operatorname{Norm}(I(w))}.
	\end{align*}
	
	If $[Z] \in \Hilb^{2}(\bbZ)$ is nonsplit, let $v \in (\bbZ[\sqrt{\mathcal{D}}])^{\oplus 3}$ be a nontrivial solution to \eqref{Eqn: SysOfEqn}. Let $i_1, i_2 \colon H^{0}(Z, \calO_{Z}) \to \bbC$ be the complex embeddings of the ring of functions on $Z$.  Set
	\[
		H_{\Le}([Z]) = \frac{|| i_{1}(v) || \cdot || i_{2}(v) ||}{\operatorname{Norm}(I(v))}.
	\]
	
\end{df}
\begin{rmk}
	Similarly to the cases where $Z$ is nonreduced or split, the expression for $H_{\Le}([Z])$ when $Z$ is nonsplit equals the Euclidean height of $[v]$.  We do not express the height in this manner because we only defined $H_{\text{Euc}}$ for $\bbQ$-points of projective space, but the definition naturally extends to any number field.
\end{rmk}

\begin{co} \label{Cor: LeRudHtComparison}
	We have 
	\begin{align} \label{Eq: LeRudComparison}
		H_{\Le}([Z])^3 = &	\ (\text{bounded function}) \cdot H_{0,3}([Z]) \\
			=& \ (\text{bounded function}) \cdot \frac{\operatorname{covol} I_{Z}(2)^3}{\operatorname{covol} I_{Z}(1)^{3}}.
	\end{align}
\end{co}

\begin{proof}
Immediate from \eqref{Eqn: CanonicalDivisor} and the  formalism of Weil's height machine.
\end{proof}

\begin{rmk}
The bounded function appearing in \eqref{Eq: LeRudComparison} is not constant.  Its behavior is nicely illustrated by closed subschemes $Z_{1}, Z_{2}, Z_{3}, \dots$ that we now define.  Given $a=1, 2, 3, \dots$, let $Z_{a} \subset \bbP^{2}_{\bbZ}$ be defined by  $\ell_{a} := a(X_{0}-3 X_{2})-(X_{1}-2 X_{2})$ and $q := (X_{0}-3 X_{2})^2$.  The only integral solutions to $\ell_{a}(x, y, z)=q(x, y, z)=0$ are  multiples of  $(3, 2, 1)$, so $H_{\Le}([Z_{a}])$ is constant as a function of $a$:
	\begin{align*}
		H_{\Le}([Z_a] ) =& (3^2+2^2+1^2) \\
					=& 14.
	\end{align*}
	By contrast, computing $\covol I_Z(2)$ using the formula \eqref{eq:covol} shows that
	\[
		H_{0,3}( [Z_a] ) = \left((2526 a^2-3204 a+1266)/(10 a^2-12 a+5)\right)^{3/2}.
	\]
	In particular, the height is bounded but nonconstant as a function of $a$.  The bounded function appearing in \eqref{Eq: LeRudComparison} is bounded by $1$ and $(13720/1263)  \cdot\sqrt{5/1263} \approx 0.6834...$. 
	\end{rmk}

\section{Lattices and the geometry of numbers}\label{sec:gon}

Upon inserting Lemma \ref{Lemma: Bijection} into the definition \eqref{eq:counting_def}, we now have the following description of $N_{s, t}(B)$, purely in terms of lattice point counting:
\begin{pr}\label{pr:geocount} We have
\begin{equation}\label{eq:geocount}
N_{s, t}(B) = \{ (\Lambda_1, \Lambda_2) \ : \ S(1) \cdot \Lambda_1 \subseteq \Lambda_2, \ \ \covol(\Lambda_1)^{s - t} \covol(\Lambda_2)^t < B \},
\end{equation}
subject to the following notations and conventions:
\begin{itemize}
\item $S$ is the polynomial ring $\Z[X_0, X_1, X_2]$, and for each $e \geq 1$ we write $S(e)$ for the lattice of polynomials of degree $e$. Note that
$S(1)$ and $S(2)$ are of ranks $3$ and $6$ respectively.
\item
$\Lambda_1 \subseteq S(1)$ and $\Lambda_2 \subseteq S(2)$ are primitive sublattices of rank $1$ and $4$ respectively, for which 
$S(1) \cdot \Lambda_1 \subseteq S(2)$.
\item
There is a natural inner product on $S(e)$, defined so that the monomials form an orthonormal basis, and we define the
covolume of a (not necessarily complete) sublattice of $\polyring(e)$ with respect to the induced volume form.
Concretely the covolume is $\sqrt{\det( \langle e_i, e_j \rangle)}$ for $e_1, \dots, e_{n_e}$ a basis for the lattice $I(e)$.
\end{itemize}
\end{pr}
The estimation of $N_{s, t}(B)$ is a purely analytic problem, for which no further algebraic geometry will be required. 

For each $\Lambda_1$, write
$F(\Lambda_1)$ for the set of lattices $\Lambda_2$ for which $S(1) \cdot \Lambda_1 \subseteq \Lambda_2$.
	Observe that $\polyring(1) \cdot \Lambda_1$ 
	is a primitive lattice in $\polyring(2)$: by the primitivity of $\Lambda_1$, we can pick a generator $l_1 \in \Lambda_1$ s.t.~there is a basis for $\polyring(1)$ of the form $l_1, l_2, l_3$.  The collection $\{ l_1\cdot l_1, l_2 \cdot l_1, l_3 \cdot l_1 \}$ generates $\polyring(1) \cdot \Lambda_1$ and extends to the basis $\{ l_i \cdot l_j \}$ of $\polyring(2)$, proving primitivity.  
	
	Therefore, the rule that sends $\Lambda_2 \in F(\Lambda_1)$  to $\Lambda_2/S(1) \cdot \Lambda_1$ defines a bijection between $F(\Lambda_1)$ and
primitive vectors in the quotient
\[
\overline{\polyring(2)}(\Lambda_1) := \polyring(2) / ( \polyring(1) \cdot \Lambda_1).
\]
	By primitivity this lattice is a free $\bbZ$-module and we endow it with an inner product so that the quotient map  $(S(1) \cdot \Lambda_1)^{\perp} \to \ovS$ is an isometry.  
 We then have the determinant relation
	\begin{equation}\label{eqn:det_rel}
		\covol(\Lambda_2) = \covol(\polyring(1) \cdot \Lambda_1) \covol( \overline{\Lambda}_{2}),
	\end{equation}
	so that for each $\Lambda_2$ and each $Y > 0$ we have
	\begin{equation}\label{eq:separate}
	\# \{ \Lambda_2 \in F(\Lambda_1) \colon \covol(\Lambda_2) < Y \} = 
			 \# \left\{ \overline{\Lambda}_2 \subset \overline{\polyring(2)}(\Lambda_1), \text{primitive, rank $1$ s.t.~}
			 \covol(\overline{\Lambda}_{2}) < \frac{Y}{|\covol(\polyring(1) \cdot \Lambda_1)|} \right\}.
\end{equation}
Our strategy for counting $N_{s, t}(B)$ is to estimate the contribution in \eqref{eq:separate} for each $\Lambda_1$ separately,
and then sum the results.

We first note the following covolume formulas, easily established by a quick computation:

\begin{lm} \label{Lemma: SL}
If $\ell = a X_0 + bX_1 + c X_2$ is a generator for $\Lambda_1$, we have
\begin{gather*}
	\covol(\Lambda_1) = (a^2 + b^2 + c^2)^{1/2}, \\
	\covol( \polyring(1) \cdot \Lambda_1 )^2  = {a^6 + 2b^2a^4 + 2c^2a^4 + 2b^4a^2 + 5c^2b^2a^2 + 2c^4a^2 + b^6 + 2c^2b^4 + 2c^4b^2 + c^6}, \\
	\frac{2}{3} (a^2 + b^2 + c^2)^3 \leq \covol( \polyring(1) \cdot \Lambda_1)^2 \leq (a^2 + b^2 + c^2)^3, \\
	\covol\big(\overline{S(2)}(\Lambda_1)\big) = \covol(S(1) \cdot \Lambda_1)^{-1}.	
\end{gather*}
\end{lm}

To count vectors in  $\overline{\polyring(2)}(\Lambda_1)$, 
we recall a basic concept from the geometry of numbers. Suppose that $\Lambda \subseteq \R^n$ is a complete (i.e., rank $n$)
lattice. Then the {\itshape successive minima} $\lambda_j$ of $\Lambda$ are the minimal real numbers $\lambda_1, \cdots, \lambda_k$ such that
$\Lambda$ contains $j$ linearly independent vectors of norm $\leq \lambda_j$. {\itshape Minkowski's second theorem} states that
the successive minima satisfy the inequalities
\begin{equation}\label{eq:mink2}
\frac{2^n}{n!} \covol(\Lambda) \leq \lambda_1 \lambda_2 \cdots \lambda_n C(n) 
\leq 2^n \covol(\Lambda),
\end{equation}
where $C(n) = \frac{ \pi^{n/2} }{\Gamma(\frac{n}{2} + 1)}$  is the volume of the unit $n$-ball. 

Our counting theorem, a variation of a lemma of Schmidt \cite{schmidt95}, illustrates that we can expect the best
results when the successive minima are roughly comparable in size:

\begin{lm} \label{Lemma: GON}
For any rank $3$ lattice 
$\Lambda \subseteq \R^3$ 
we have
\beq\label{eq:GON}
N := \# \{ v \in \Lambda - \{ 0 \} \ : \ |v| < R, \ v \ \textnormal{primitive} \} =
 \frac{4 \pi}{3 \zeta(3)}  \cdot \frac{R^3}{|\covol(\Lambda)|} + O\left( 
\log^*(R/\lambda_1) \cdot \frac{\lambda_2 \lambda_3 R}{|\covol(\Lambda)|} +  \frac{\lambda_3 R^2}{|\covol(\Lambda)|} \right),
\eeq
where $\log^*(t) := \max(1, \log(t))$.
\end{lm}

\begin{proof}
This is a variation of \cite[Lemma 1]{schmidt95} and we partially follow the proof given there. For now
let $\Lambda$ be a complete lattice of rank $n$; we specialize to $n = 3$ later.

Let $\lambda_i$ be the successive minima of $\Lambda$. By \cite[p. 135, Lemma 8]{Cassels} there is a basis $v_1, \dots, v_n$ of $\Lambda$
with $v_i \in i \lambda_i \calB$, where $\calB$ is the closed unit ball. We choose the $v_i$ such that
$|v_1| \leq |v_2| \leq \dots \leq |v_n|$ so that $|v_i| \geq \lambda_i$ for each $i$, and we obtain a version of \eqref{eq:mink2}
for this basis, namely
\begin{equation}\label{eq:mink3}
\frac{2^n}{n!} \covol(\Lambda) \leq |v_1| |v_2| \cdots |v_n| C(n) 
\leq 2^n n! \covol(\Lambda),
\end{equation}

If $v = a_1 v_1 + \dots + a_n v_n$ is any $\R$-linear combination of the $v_i$ with $v \in B$, then we claim that
$|a_i v_i| \leq \frac{2^n n!}{C(n)} |v|$ for each $i$.  This is established by the following computation (for each $i$ with $a_i \neq 0$):
\begin{align*}
|v_1| |v_2| \cdots |v_n| \cdot \frac{C(n)}{2^n n!} \leq & 
\covol(\Lambda) \\
= & \ \covol(v_1, \dots, v_i, \dots, v_n) \\
= & \ \frac{1}{|a_i|} \covol(v_1, \dots, v_{i - 1}, a_1 v_1 + \dots + a_n v_n, v_{i + 1}, \dots, v_n) \\
\leq & \ \frac{1}{|a_i|} |v_1| \dots |v_{i - 1}| \cdot |v_{i + 1}| \dots |v_n| \cdot |v|.
\end{align*}

\medskip
Now, let $\tau \ : \R^n \rightarrow \R^n$ be the linear map with $\tau(v_i) = e_i$, where $e_i \in \Z^n$ is the
$i$th standard basis vector $(0, \cdots, 1, \cdots, 0)$. Thus $\tau(\Lambda) = \Z^i$ and $\tau(\calB) = \calE$, where $\calE$ is an ellipsoid of volume
$V(\calE) = V(\calB) / \covol(\Lambda)$. The previous computations establish that 
\begin{equation}\label{eq:dav_bound}
\calE \subseteq \frac{C(n)}{2^n n!} \cdot \Big( \big[ - \frac{1}{\lambda_1}, \frac{1}{\lambda_1} \big] \times \cdots \times
\big[- \frac{1}{\lambda_n}, \frac{1}{\lambda_n} \ \big] \Big).
\end{equation}

Let $N'$ be the number of points in $\Lambda \cap R \calB$, or equivalently in $\Z^n \cap R \calE$, without any primitivity condition.
We invoke Davenport's lemma \cite{dav_lemma}, which states that
\[
|N' - \textnormal{vol}(R \calE)| \ll \max_{\calE'} \textnormal{vol}(R \calE'),
\]
where $\calE'$ ranges over the projections of $\calE$ onto all coordinate planes of dimension $< n$, and where the volume of the $0$-dimensional
projection is understood to be $1$. The equation \eqref{eq:dav_bound} establishes suitable bounds on their volumes. For simplicity specializing now to $n = 3$,
we obtain
\[
N' = \frac{C(3)}{\covol(\Lambda)} R^3 + 
O\left(\max\left(1, \frac{\lambda_2 \lambda_3 R}{|\covol(\Lambda)|}, \frac{\lambda_3 R^2}{|\covol(\Lambda)|}\right) \right).
\]

To count {\itshape primitive} lattice points, we use M\"obius inversion, applied to dilates of the lattices above. We obtain
\[
N = \sum_{d = 1}^{\lfloor R/\lambda_1 \rfloor} \left( \mu(d) \frac{C(3)}{\covol(\Lambda)} (R/d)^3 
+
O\left(\max\left(1, \frac{\lambda_2 \lambda_3 (R/d)}{|\covol(\Lambda)|}, \frac{\lambda_3 (R/d)^2}{|\covol(\Lambda)|}\right) \right) \right).
\]
Note that $1 \ll \frac{\lambda_2 \lambda_3 (R/d)}{|\covol(\Lambda)|}$ whenever $d \leq R/\lambda_1$, so we may eliminate
the first of the error terms. We make an error $\ll \frac{R^3}{\covol(\Lambda) (R/\lambda_1)^2} = \frac{R \lambda_1^2}{\covol(\Lambda)}$ in extending the summation
over $d$ in the main term to an infinite sum. 
 We therefore obtain
\begin{equation}\label{eq:almost_done}
N = \frac{C(3)}{\zeta(3) \covol(\Lambda)} R^3 + O\left( 
\log^*(R/\lambda_1) \cdot \frac{\lambda_2 \lambda_3 R}{|\covol(\Lambda)|} +  \frac{\lambda_3 R^2}{|\covol(\Lambda)|} \right).
\end{equation}

\end{proof}

\begin{rmk}
Using a zeta function approach, the main result of \cite{LDTT} yields a better error term of
$O\Big(\frac{\lambda_3^{3/2} R^{3/2}}{|\covol(\Lambda)|}\Big)$ in \eqref{eq:GON}. Although this yields a modest improvement in the error term in 
Theorem \ref{thm:main} (for some values of $\frac{s}{t}$), we still would obtain asymptotic estimates for $N_{s, t}(B)$ precisely when $\frac{s}{t} > 1$.
Since the paper \cite{LDTT} involves complicated analysis with Bessel functions, 
we chose the above, more self-contained approach.
\end{rmk}

We will apply the proceeding lemma to the $3$-dimensional lattice $\ovS := \polyring(2)/(\polyring(1) \cdot \Lambda_1)$. 
Write $\lambda_1$, $\lambda_2$, and $\lambda_3$ for the successive minima of this lattice, where each $\lambda_i$
is the length of a vector $y_i \in \ovS$. 
Since the map $(\polyring(1) \cdot \Lambda_1)^{\perp} \to \ovS$ is an isometry, these successive minima also equal
the distances of lifts $\widetilde{y_i} \in \polyring(2)$ to the subspace $V$ described in the statement
of Lemma \ref{lm:minima}.

As we just
saw, we obtain large error terms from $\Lambda_1$ for which $\lambda_1$ is very small, and so we prove the following upper bounds on it:

\begin{lm}\label{lm:sm_bounds}
For each rank one lattice $\Lambda_1 \subseteq S(1)$ spanned by a primitive vector $aX_0 + bX_1 + cX_2$, 
write $\lambda_1 \leq \lambda_2 \leq \lambda_3$ (or $\lambda_1(\Lambda_1)$, etc.)
for the successive minima of $\ovS$, and write $M = \max(|a|, |b|, |c|)$. 
These quantities satisfy the following:
\begin{enumerate}[(a)]
\item
We have
$\lambda_1 \lambda_2 \lambda_3 \asymp M^{-3}$. 
\item
We have $\lambda_3 \leq 1$.
\item
We have $\lambda_1 \geq \frac{1}{7} M^{-2}$, and hence (by (a)) $\lambda_2 \lambda_3 \ll M^{-1}$.
\end{enumerate}
\end{lm}

\begin{proof}

We have, by construction, that
\begin{equation}\label{eq:vol_triv}
\covol(\ovS) \cdot \covol(S(1) \cdot \Lambda_1) = \covol(S_2) = 1,
\end{equation}
so that $\covol(\ovS) \asymp M^{-3}$ by Lemma \ref{Lemma: SL}.
Thus (a) follows from Minkowski's theorem
\eqref{eq:mink2}.

The claim of (b) is the `trivial bound': the monomials $X_0^2$, $X_0X_1$, $X_0X_2$, $X_1^2$, $X_1X_2$, $X_2^2$ span $S(2)$, and each monomial is
a distance $1$ from $0 \in S(2)$, hence a distance at most $1$ from the subspace generated by 
$S(1) \cdot \Lambda_1$. The images of these monomials 
$\pmod{S(1) \cdot \Lambda_1}$ span $\ovS$, and each has length at most $1$. Therefore, the successive minima are all bounded by $1$.

It remains to prove (c); choosing the monomial basis above for $S(2) \simeq \Z^6$, the claim is equivalent to the following lemma.
\end{proof}

\begin{lm}\label{lm:minima}
For integers $a$, $b$, $c$ not all sharing a common factor, write $M = \max(|a|, |b|, |c|)$ and 
let $V := V(a, b, c) \subseteq \R^6$ be the subspace spanned by $v_1 := (a, b, c, 0, 0, 0)$, $v_2 := (0, a, 0, b, c, 0)$, and $v_3 := (0, 0, a, 0, b, c)$.

Then, if $x = (x_1, x_2, x_3, x_4, x_5, x_6) \in \Z^6 - V$, we have $\dist(x, V) \geq \frac{1}{7M^2}$.
\end{lm}

\begin{proof}
Given $(a, b, c)$, we may assume by symmetry that $|a| \geq |b| \geq |c|$. We may also assume that $ab \neq 0$;
otherwise the $v_i$ would all be parallel to the coordinate axes. 
For the moment, we also assume that $c \neq 0$.

Choose vectors
\begin{equation}
x = x(a, b, c) = (x_1, x_2, x_3, x_4, x_5, x_6) \in \Z^6 - V
\end{equation} and
\begin{equation}\label{eq:wp}
w' = (w'_1, w'_2, w'_3, w'_4, w'_5, w'_6) = \gamma'_1 v_1 + \gamma'_2 v_2 + \gamma'_3 v_3 \in V
\end{equation}
so that the (nonzero) distance $\alpha := |w' - x|$ is minimized. We may assume that $\alpha < \frac{1}{7M^2} < \frac12$ (if not, we're done). This implies that
the integer nearest to each $w'_i$ is $x_i$.

We have $|w'_i - x_i| \leq \alpha$ for each $i$, and 
the coefficients $w'_1$, $w'_4$, and $w'_6$ are determined exclusively by $\gamma'_1$, $\gamma'_2$, and $\gamma'_3$ respectively.
We now choose $\gamma_1, \gamma_2, \gamma_3$ so that the coefficients $w_1$, $w_4$, and $w_6$ of the vector
\begin{equation}\label{eq:w}
w = (w_1, w_2, w_3, w_4, w_5, w_6) = \gamma_1 v_1 + \gamma_2 v_2 + \gamma_3 v_3
\end{equation}
equal $x_1$, $x_4$, and $x_6$ respectively, so that we have
\[
w = \left( x_1, \frac{x_1 b}{a} + \frac{x_4 a}{b}, \frac{x_1 c}{a} + \frac{x_6 a}{c}, x_4, \frac{x_4 c}{b} + \frac{x_6 b}{c}, x_6 \right).
\]
We have $|\gamma_1 - \gamma'_1| \leq \frac{\alpha}{|a|}$,  
$|\gamma_2 - \gamma'_2| \leq \frac{\alpha}{|b|}$, and $|\gamma_3 - \gamma'_3| \leq \frac{\alpha}{|c|}$, so that 
$|\gamma_1 v_1 - \gamma'_1 v_1| < \sqrt{3} \alpha$, 
$|\gamma_2 v_2 - \gamma'_2 v_2| < \sqrt{3} \alpha \cdot \frac{|a|}{|b|}$, and
 $|\gamma_3 v_3 - \gamma'_3 v_3| < \sqrt{3} \alpha \cdot \frac{|a|}{|c|}$. Therefore, we obtain
 the inequalities
\begin{equation}\label{eqn:ac}
0 < |w - x| \leq |w' - x| + |w' - w| \leq \alpha + 3\sqrt{3} \alpha \frac{|a|}{|c|} < 7 \alpha \frac{|a|}{|c|}.
\end{equation}

If both $w_3$ and $w_5$ are integers, then we have $ac \mid x_1 c^2 + x_6 a^2$
and hence $c \mid x_6 a^2$, and similarly $c \mid x_6 b^2$. Since $a, b, c$ do not all have a common factor,
we have $c \mid x_6$.
Because $ac \mid x_1 c^2 + x_6 a^2$ and $c \mid x_6$, we get  $a \mid x_1 c$, and we similarly see that 
$b \mid x_4 c$. 
Since $(a, c) \cdot (b, c) \leq |c|$, we have 
$(x_1, a) \cdot (x_4, b) \geq \frac{|ab|}{|c|}$. But then the fraction 
$\frac{x_1 b}{a} + \frac{x_4 a}{b}$ (which is not an integer) has denominator which divides
$\frac{|a|}{(x_1, a)} \cdot \frac{|b|}{(x_4, b)} \leq |c|$, so that $7 \alpha \frac{|a|}{|c|} \geq \frac{1}{|c|}$ and we are finished.

Alternatively, if either of $w_3$ and $w_5$ is not an integer, we have $|w - x| \geq \frac{1}{|ac|}$, and hence 
by \eqref{eqn:ac} that $\alpha > \frac{1}{7|a|^2} \geq \frac{1}{7M^2}$.

\medskip
{\itshape If $c = 0$.} Finally, if $c = 0$, we vary the above argument as follows. In \eqref{eq:wp}, now
$w_1'$, $w_4'$, and $w_3'$ respectively determine our choices of $\gamma_1$, $\gamma_2$, and
$\gamma_3$, and we obtain 
\[
w = \left( x_1, \frac{x_1 b}{a} + \frac{x_4 a}{b}, x_3, x_4, \frac{b}{a} x_3, 0 \right)
\in V
 \]
 with $0 < |w - x| < 7\alpha\frac{|a|}{|b|}$. Since either $w_2 \not \in \Z$ or $w_5 \not \in \Z$, we therefore
 have $|w - x| \geq \frac{1}{|ab|}$ and arrive at the same conclusion as before.
\end{proof}

A natural question is whether the results of Lemma \ref{lm:sm_bounds} can be improved. The following proposition shows that the answer is essentially `no'.
(It might still be possible to obtain improvements that hold for `most' $(a, b, c)$; such a result might yield an improvement to our main result.)

\begin{pr} \label{Lemma: BestBounds}
	The bounds in Lemma~\ref{lm:sm_bounds} are the best possible in the sense that
	\begin{gather}
		\operatorname{lim sup} \lambda_{1}(\Lambda_{1}) = 1 \text{ and } \label{eq:optimal1} \\
		\operatorname{lim inf} \lambda_{1}(\Lambda_{1})/M^{-2} < \infty. \label{eq:optimal2}
	\end{gather}
\end{pr}

\begin{proof}
To prove \eqref{eq:optimal1}, consider any $\Lambda_1$ generated by a vector of the form $v := aX_0 + bX_1$, with no $X_2$ term. Then the vector $X_2^2 \in S(2)$ has distance at least $1$ from 
the subspace generated by 
$S(1) \cdot \Lambda_1$, so the result follows from Lemma \ref{lm:sm_bounds}(b).

To prove \eqref{eq:optimal2}, consider $\Lambda_1$ and $v$ of the same form.
 If $\textnormal{gcd}(a, b) = 1$, then the equation $a^2 t - b^2 s = 1$ has an integer solution $(s, t)$. We have 
 	\[
		\frac{1}{a b} X_0 X_1= (s X_0^2- t X_1^2) + \frac{t}{b} (a X_0 + b X_1) X_1 - \frac{s}{a}(a X_0 + b X_1)X_0,
	\]
	so that the vector $-sX_0^2 + tX_1^2$ has distance at most $1/ab$ from the subspace generated by $S(1) \cdot \Lambda_1$. Choosing
	$a = M$, $b = M - 1$, we get $\lambda_1(\Lambda_1) \leq 1/(M^2 - M)$, proving \eqref{eq:optimal2}, and indeed showing that
	$1 \ge \operatorname{lim inf} \lambda_{1}(\Lambda_{1})/M^{-2} \ge 1/7$.
\end{proof}

\section{Proof of Theorem \ref{thm:main}}\label{sec:proof}
We can now assemble our ingredients into a proof of our main result.
By Proposition \ref{pr:geocount} and \eqref{eq:separate}, 
		we have (for each $s, t > 0$)
\begin{align*}\label{eq:mt_basic}
		N_{s, t}(B) & = \# \{ [Z] \in \Hilb^{2}(\bbZ) \colon H_{s,t}([Z]) < B \} \\
		& = \sum_{\Lambda_1} \# \{ \Lambda_2 \in F(\Lambda_1) \colon \covol(\Lambda_1)^{s - t} 
		\covol(\Lambda_2)^t < B\} \\ & =  
		 \sum_{\Lambda_1} \# \{ \Lambda_2 \in F(\Lambda_1) \ : \ \covol(\Lambda_2) < B^{1/t} (\covol \Lambda_1)^{1 - s/t} \} \nonumber
		 \\ & =  
		 \sum_{\Lambda_1} 
		  \# \left\{ \overline{\Lambda}_2 \subset \overline{\polyring(2)}(\Lambda_1), \text{primitive, rank $1$ s.t.~}
			 \covol(\overline{\Lambda}_{2}) < \frac{B^{1/t} (\covol \Lambda_1)^{1 - s/t}}{\covol(\polyring(1) \cdot \Lambda_1)} \right\}.
\end{align*}
For each $\Lambda = \Lambda_1$ in the sum, generated by $\pm (aX_0 + bX_1 + cX_2)$ with 
$M = M(\Lambda) = \max(|a|, |b|, |c|)$, write $\lambda_{1, \Lambda} \leq \lambda_{2, \Lambda} \leq \lambda_{3, \Lambda}$
for the successive minima in the quotient lattice $\overline{S(2)}(\Lambda_1)$. By Lemma \ref{Lemma: GON}, we have
	\begin{multline*}
		\# \left\{ \overline{\Lambda}_2 \subset \overline{\polyring(2)}(\Lambda_1), \text{primitive, rank $1$ s.t.~}
			 \covol(\overline{\Lambda}_{2}) < Y \right\}
			 = \\ 
		\nonumber
		\begin{cases}
		\covol(S(1) \cdot \Lambda_1) \cdot \Big(
		\frac{2 \pi}{3 \zeta(3)}
		Y^3 +
		O\left(\lambda_{3, \Lambda} \lambda_{2, \Lambda} Y \log^*(Y/\lambda_1)
		 + \lambda_{3, \lambda} Y^2 \right) \Big) \ \textnormal{unconditionally}, \\
		0 \ \ \ \textnormal{ if $\lambda_{1, \Lambda} \geq Y$}. \end{cases}
		\end{multline*}
	
By Lemma \ref{Lemma: SL} we have
\begin{equation}
\covol \Lambda_1 \asymp M, \ \ \covol(S(1) \cdot \Lambda_1) \asymp M^3,
\end{equation}
and by Lemma \ref{lm:sm_bounds} the successive minima satisfy 
\begin{equation}\label{eq:sm_est}
\lambda_{1, \Lambda} \lambda_{2, \Lambda} \lambda_{3, \Lambda} \asymp M^{-3}, \ \ \ 
\lambda_{1, \Lambda} \gg M^{-2}, \ \ \ \lambda_{2, \Lambda} \lambda_{3, \Lambda} \ll M^{-1}, \ \ \ \lambda_{3, \Lambda} \leq 1.
\end{equation}

Assembling all of this, we conclude that
\begin{equation}
N_{s, t}(B) =  \frac{1}{2} \sum_{M(a, b, c) \ll B^{1/s}} 
\left( \frac{2 \pi}{3 \zeta(3)} \cdot \frac{B^{\frac{3}{t}} (\covol \Lambda)^{3 - 3\frac{s}{t}}}{(\covol(\polyring(1) \cdot \Lambda))^2} +
O\left( B^{\frac{1}{t}} M^{- \frac{s}{t}} \log(BM) + 
B^{\frac{2}{t}} M^{-1 - 2 \frac{s}{t}}
\right) \right), \label{eq:main_sum}
\end{equation}
so that

\begin{align}
N_{s, t}(B) \ - \ &
\frac{\pi B^{3/t}}{3 \zeta(3)} \sum_{a, b, c} \frac{ (a^2 + b^2 + c^2)^{\frac{3}{2} - \frac{3}{2} \cdot \frac{s}{t}} }
{a^6 + 2b^2a^4 + 2c^2a^4 + 2b^4a^2 + 5c^2b^2a^2 + 2c^4a^2 + b^6 + 2c^2b^4 + 2c^4b^2 + c^6} \label{eqn:mainterm} \\
\ll \ & \ E_1 + E_2 + E_3, \nonumber
\end{align}
for error terms $E_1, E_2, E_3$ to be described shortly, and 
where the sum in \eqref{eqn:mainterm} is over all primitive triples $(a, b, c) \in \Z^3$ and converges absolutely 
for $\frac{s}{t} > 0$.

The first of our three error terms is the tail of the main term \eqref{eqn:mainterm}; we have
\begin{align*}
E_1 := & \ B^{3/t} \sum_{\substack{a, b, c \\ M \gg B^{1/s}}} \frac{ (a^2 + b^2 + c^2)^{\frac{3}{2} - \frac{3}{2} \cdot \frac{s}{t}} }
{a^6 + 2b^2a^4 + 2c^2a^4 + 2b^4a^2 + 5c^2b^2a^2 + 2c^4a^2 + b^6 + 2c^2b^4 + 2c^4b^2 + c^6} \\
\ll & \  B^{3/t} \sum_{\substack{a, b, c \\ M \gg B^{1/s}}} M^{-3 - 3 \frac{s}{t}} \\
= & \ O(1),
\end{align*}
where the sum over $a, b, c$ can be bounded as follows: for each $T > 1$, there are $O(T^3)$ triples $a, b, c$ with
$M = \max(|a|, |b|, |c|) \in [T, 2T]$; we break our sum up into such dyadic intervals, and then sum over $T$.

The error terms $E_2$ and $E_3$ correspond to the two error terms in \eqref{eq:main_sum}, and they can be bounded similarly. We have

\[
E_2 
\ll B^{1/t} \log(B)
\sum_{|a|, |b|, |c| \ll B^{1/s}}
M^{- \frac{s}{t}} \ll \begin{cases}
B^{1/t} \log(B) & {\text if } \ \frac{s}{t} > 3,\\
B^{1/t} \log(B)^2 & {\text if } \ \frac{s}{t} = 3,\\
B^{3/s} \log(B) & {\text if } \ \frac{s}{t} < 3,
\end{cases}
\]
and

\[
E_3 \ll \
B^{2/t} 
\sum_{|a|, |b|, |c| \ll B^{1/s}}
M^{-1 - 2 \frac{s}{t}} \ll
\begin{cases}
B^{\frac{2}{t}} & {\text if } \ \frac{s}{t} > 1,\\
B^{\frac{2}{t}} \log(B) & {\text if } \ \frac{s}{t} = 1,\\
B^{\frac{2}{s}} & {\text if } \ \frac{s}{t} < 1.
\end{cases}
\]

We thus have
\begin{equation}
E_1 + E_2 + E_3 \ll 
\begin{cases}
B^{\frac{2}{t}} & {\text if } \ \frac{s}{t} > \frac32,\\
B^{\frac{3}{s}} \log(B) & {\text if } \ \frac{s}{t} \leq \frac32,
\end{cases}
\end{equation}
smaller than the main term whenever $\frac{s}{t} > 1$.

\section{The Batyrev--Manin conjecture} \label{Section: BatyrevManin}
We complete our analysis of $N_{s, t}(B)$ by verifying the following:
\begin{co} \label{Corollary: baryrevManin}
	For $\frac{s}{t} > 1$, the height function $H_{s,t}$ on $\Hilb^{2}$ satisfies the Batyrev--Manin conjecture.
\end{co}
There are several forms of the conjecture, and in  Corollary~\ref{Corollary: baryrevManin}, we mean the following form.  Given an ample line bundle $D$, define the Nevanlinna constant $\alpha(D)$ by $\alpha(D) = \operatorname{inf} \{ t \in \bbR \colon D^{\otimes t} \otimes K \text{ is effective} \}$ for $K$ the canonical line bundle.  Then we take the conjecture to be that the rational points on $\Hilb^{2}$ satisfy
\begin{equation} \label{Eqn: baryrevManin}
	N_{D}(B) \sim c B^{\alpha} \log(B)^{\beta}
\end{equation}
for some constant $c \in \bbR$, $\alpha = \alpha(D)$, and $\beta$ the codimension of the unique face of the effective cone containing $D^{\otimes \alpha} \otimes K$. (In general, one only requires that \eqref{Eqn: baryrevManin} holds after possibly passing to a finite field extension and removing a thin set, but our results  show is this is unnecessary for $\Hilb^{2}(\bbP^{2})$.)

\begin{proof}[Proof of Corollary~\ref{Corollary: baryrevManin}]
	We will prove the stronger statement that $\alpha = 3/t$ and $\beta=0$ for  $D_{s,t} = D_{1}^{\otimes s-t} \otimes D_{2}^{\otimes t}$ with $s, t >0$.  The effective cone is known to be the cone of the line bundles $D_{s,t}$ with $t, s+t \ge 0$ (by \cite[Theorem~4.5]{arcara13}; that description involves a different basis, but use \cite[Proposition~3.1]{arcara13} to rewrite this in terms of $D_1$, $D_2$).  Recall from Equation~\eqref{Eqn: CanonicalDivisor} that the canonical bundle is $K = D_{0,-3}$, so
	\begin{align*}
		D_{s, t}^{\otimes r} \otimes K =& D_{r s, r t-3}.
	\end{align*}
	This is effective when $r t-3 \ge 0$.  In other words, $\alpha = 3/t$.  The class $ D_{s,t}^{\otimes \frac{3}{t}} \otimes K =D_{3 \frac{s}{t},0}$ lies on the codimension $0$ face spanned by $D_1=D_{1,0}$, so $\beta=0$.
\end{proof}

\section{The associated quadratic algebra} \label{Section: QuadAlgebra}
In the present section, we explore the connection between integral points of $\Hilb^{2}$ and quadratic rings.  In Section~\ref{Section: BackgroundHilb}, we defined $\Hilb^{2}(\bbZ)$ to be the set of families $Z \subset \bbP^{2}_{\bbZ}$ of degree $2$, dimension $0$ closed subschemes (i.e.~subschemes that are $\bbZ$-flat with fibers of degree $2$ and dimension $0$).  Such a closed subscheme determines a quadratic ring, the ring of regular functions $A=H^{0}(Z, \calO_{Z})$.  A natural measure of a quadratic ring is its discriminant, and the main result of this section is a comparison of the discriminant of $A$ with height functions on $\Hilb^2$. The discriminant is not a height function, but we show in Proposition~\ref{Prop: DiscBound} that it is bounded above by the height function $H_{-2, 2}$.  

To study $A$ in depth, we need to recall some facts about the Proj construction from \cite[II, Section~2]{hartshorne77}.  Suppose somewhat generally that $R = R(0) \oplus R(1) \oplus R(2) \dots$ is a graded ring, and consider $\operatorname{Proj} R$.  This scheme has the property that a homogeneous element $f \in R$ determines an open affine scheme $U_{f}$ that is isomorphic to the spectrum of $\left(R[1/f]\right)(0)$.  Here $\left(R[1/f]\right)(0)$ is the subring of $R[1/f]$ consisting of the homogeneous elements of degree $0$.  A collection of elements $\{ f_1, \dots, f_n \}$ has the property that the open subsets $U_{f_1}, \dots, U_{f_n}$ cover $\operatorname{Proj} R$ provided the ideal   $(f_1, \dots, f_n)$ contains a power of the irrelevant ideal $R(1) \oplus R(2) \oplus R(3) \oplus \dots$. 

The elements $[Z]$ of  $\Hilb^{2}(\bbZ)$ are all constructed by the Proj construction.  Recall that Lemma~\ref{Lemma: Bijection} states that every such $Z$ can be written as
\begin{equation} \label{Eqn: RestateEqnForSubsch}
	Z = \operatorname{Proj} \frac{\bbZ[X_0, X_1, X_2]}{\bbZ \cdot \ell + \bbZ \cdot q},
\end{equation}
where as before $\ell \in S(1)$ is a primitive linear polynomial and $q \in S(2)$ a quadratic polynomial such that $X_0 \ell, X_1 \ell, X_2 \ell, q$ span a primitive lattice.  We use this description to compute $H^{0}(Z, \calO_{Z})$.

\begin{lm} \label{Lm: DescribeAlgebra}
	Suppose that $[Z] \in \Hilb^{2}(\bbZ)$ is defined by $\ell$ and $q$ as in Equation~\eqref{Eqn: RestateEqnForSubsch} and let $v, w \in \bbZ^{\oplus 3}$ be a basis for $\{ v \in \bbZ^{\oplus 3} \colon \ell(v)=0 \}$.  Then the ring of functions $A = H^{0}(Z, \calO_{Z})$ on $Z$ is the unique quadratic order with discriminant equal to the discriminant of $q(S v+T w)  \in \bbZ[S, T]$.
	\end{lm}
\begin{proof}
	Write $\mathcal{D}_1$ for the discriminant of $A$ and $\mathcal{D}_{2}$ for the discriminant of $q(S v+T w)$. The triple $S v + T w$ defines an isomorphism $\operatorname{Proj} \bbZ[S, T] = \bbP^{1}_{\bbZ} \cong \{ \ell=0\} \subset \bbP^{2}_{\bbZ}$, and under this isomorphism, $Z$ corresponds to $\operatorname{Proj} \frac{\bbZ[S, T]}{q(S v+T w)}$.   For the remainder of the proof, we will work directly with this last scheme.
	
	Write
	\[
		q(S v+T w) = a S^2 + b S T + c T^2.
	\]
	By definition, $\mathcal{D}_{2}=b^2-4 a c$, so we need to compute $\mathcal{D}_{1}$  to be the same.  
	
	Consider the open affine subscheme $U_{f_1} \subset Z$ for $f_{1} = a T$.  The scheme $U_{f_1}$ is isomorphic to  the spectrum of  $\frac{\bbZ[r, 1/a]}{ a r^2 + b r + c}$ for $r = S/T$.  This ring is free of rank $2$ over $\bbZ[1/a]$, so it has a well-defined discriminant.  By using the basis $1, r$, we get that the discriminant is $(b^{2}-4 a c)/a^2 = \mathcal{D}_{2}/a^2$.  This discriminant is, however,  only defined up to multiplication by a unit that is a square, so we can only conclude that
	\begin{equation*} \label{Eqn: DiscrEqnFirst}
		\mathcal{D}_{1} = \mathcal{D}_{2} \cdot a^{2 k} \text{ for some $k$.}
	\end{equation*}
	
	To complete the argument, we replace the  use of  $U_{f_1}$ with
	\begin{gather*}
		U_{f_2} = \Spec \frac{\bbZ[r, 1/c]}{ a + b r + c r^{2}} \text{ for } f_2 = c S, r = T/S \text{ and } \\
		U_{f_3} = \Spec \frac{\bbZ[ r, 1/(a+b+c)]}{ a +(2 a+b)  r + (a+b+c) r^{2}} \text{ for } f_{3}= (a+b+c)(S-T), r = T/(S-T).
	\end{gather*}
	We conclude that
	\begin{align}
		\mathcal{D}_{1} =& \mathcal{D}_{2}\cdot c^{2 l} \text{ for some $l$ and}   \label{Eqn: DiscrEqnSecond} \\
		=& \mathcal{D}_{2} \cdot (a+b+c)^{2 m} \text{ for some $m$.} \label{Eqn: DiscrEqnThird}
	\end{align}
	The integers $a, c, a+b+c$ must be relatively prime (because we assumed  $X_0 \ell, X_1 \ell, X_2 \ell, q$ span a primitive lattice).  Thus Equations \eqref{Eqn: DiscrEqnFirst}, \eqref{Eqn: DiscrEqnSecond}, and \eqref{Eqn: DiscrEqnThird} imply  $\mathcal{D}_{1}=\mathcal{D}_{2}$ as desired.
\end{proof}

For simplicity we write $\operatorname{Disc}(Z) := \operatorname{Disc}(H^{0}(Z, \calO_{Z}))$.

In Definition~\ref{Def: NonredSplitNonsplit}
we characterized $Z$ as being nonreduced, split, or nonsplit depending on the rationality of the solutions to $\ell(x, y, z)=q(x, y, z)=0$.  Lemma~\ref{Lm: DescribeAlgebra} connects that definition to usual use of the terms in algebraic number theory.

\begin{lm}
	In the sense of Definition~\ref{Def: NonredSplitNonsplit}, the subscheme $Z$ is:
	\begin{itemize}
	\item nonreduced, if $\operatorname{Disc}(Z)=0$, or equivalently if $H^{0}(Z, \calO_{Z}) \otimes_{\bbZ} \bbQ$ is isomorphic to $\bbQ[\epsilon]/\epsilon^2$;
	\item split, if $\operatorname{Disc}(Z)$  is a  nonzero perfect square, or equivalently if $H^{0}(Z, \calO_{Z}) \otimes_{\bbZ} \bbQ$ is isomorphic to $\bbQ \times \bbQ$;
	\item nonsplit, if $\operatorname{Disc}(Z)$  is not a perfect square, or equivalently if $H^{0}(Z, \calO_{Z}) \otimes_{\bbZ} \bbQ$ is a quadratic field.
\end{itemize}
\end{lm}
\begin{proof}
	In the notation of Definition~\ref{Def: NonredSplitNonsplit}, $Z$ is nonreduced, split, or nonsplit depending on whether $q(S v+T w)$ has 1, 2, or 0 linear factors over $\bbQ$.  
	The conditions on $\operatorname{Disc}(Z)$ follow from computing the discriminant in terms of the $\overline{\bbQ}$-factorization into linear factors, and the last conditions
	follow as	\[
		H^{0}(Z, \calO_{Z}) \otimes_{\bbZ} \bbQ \cong \bbQ[r]/ \big( r^2 - \operatorname{Disc}(Z) \big).
	\]
\end{proof}

With a view towards bounding the discriminant by a height function, we proceed to derive explicit expressions for the discriminant starting with the split case.

\begin{co} \label{Cor: ExplicitDiscriminantSplit}
	Suppose that  $[Z] \in \Hilb^{2}(\bbZ)$ is split.  Let $v=(v_0, v_1, v_2)$ and $w=(w_0, w_1, w_2) \in \bbZ^{3}$ be linearly independent primitive vectors that solve
	\[
		q(x, y, z) = \ell(x, y, z)=0.
	\]
	Then the discriminant of $H^{0}(Z, \calO_{Z})$ is 
	\begin{equation} \label{Eqn: GCD}
		\operatorname{Disc}(Z) = \operatorname{GCD}(v_{1} w_{2}-v_{2} w_{1}, v_{2} w_{0} - v_{0} w_{2}, v_{0} w_{1}-v_{1} w_{0})^{2}.
	\end{equation}
\end{co}
\begin{proof}
	The discriminant and the greatest common divisor are invariant under $\operatorname{GL}_3(\Z)$-change of coordinates, so we can assume that $v=(1,0,0)$ and $w=(w_0, w_1,0)$ with $w_1 \ne 0$ and $\operatorname{GCD}(w_0, w_1)=1$ (transform the primitive lattice generated by $v$ and $w$ to the lattice spanned by $(1,0,0)$ and $(0,1,0)$; to see that the greatest common divisor is invariant, observe that the integers appearing in \eqref{Eqn: GCD} are the components of the cross product $v \times w$).  
	
	For $v=(1,0,0)$ and $w=(w_0, w_1,0)$, the greatest common divisor appearing in \eqref{Eqn: GCD} is $w_1$, so we need to show that $\operatorname{Disc}(Z) = (w_{1})^2$.  
	
	For our choice of $v$ and $w$, we have $\ell = X_2$ and $q = w_1 X_0 X_1 - w_0 X_1^2$, and a basis for $\{ (x, y, z) \in \bbZ^{\oplus 3} \colon \ell(x,y,z)=0\}$ is $e=(1,0,0)$ and $f=(0,1,0)$. 
	With this basis, we have
	\[
		q(S e+ T f) = T (w_{1} S - w_{0} T),
	\]
	and this polynomial has the desired discriminant. 
\end{proof}

We now turn to the nonsplit case.

\begin{lm} \label{Lemma: VectorInCoor}
	Suppose that $[Z] \in \Hilb^{2}(\bbZ)$ is nonsplit with discriminant $\mathcal{D}$.  Let  $v=(a_{1} + a_{2} \sqrt{\mathcal{D}}, b_{1}+b_{2} \sqrt{\mathcal{D}}, c_{1} + c_{2} \sqrt{\mathcal{D}}) \in (\bbZ[\sqrt{\mathcal{D}}])^{\oplus 3}$  be a nontrivial solution to 
	\[
		q(x, y, z) = \ell(x, y, z)=0.
	\]
	 Then there exists a basis $e, f$ for the lattice $\{ (x, y, z) \in \bbZ^{\oplus 3} \colon \ell(x, y, z)=0 \}$ such that
	\begin{align}
		(a_1, b_1, c_1) =& g e \text{ and } \label{Eqn: VectorInCoord} \\
		(a_2, b_2, c_2) =& \alpha e + \beta f \text{ for $g, \alpha, \beta \in \bbZ$.} 
	\end{align}
\end{lm}
\begin{proof}
	Observe that both $(a_1, b_1, c_1)$ and $(a_2, b_2, c_2)$ lie in $\{ (x, y, z) \in \bbZ^{3} \colon \ell(x, y, z)=0 \}$ as $\ell$ has integral coefficients.  Furthermore, $(a_1, b_1, c_1)$ must be nonzero (since otherwise the system of equations would have a primitive integral solution, contradicting the assumption that $Z$ is nonsplit).  Thus if $g = \operatorname{GCD}(a_1, b_1, c_1)$, then $e := (a_1/g, b_1/g, c_1/g)$ is a primitive vector, so it can be extended to a basis $e, f$ for $\{ (x, y, z) \in \bbZ^{3} \colon \ell(x, y, z)=0 \}$.  This basis has the desired properties.
\end{proof}

\begin{co}  \label{Cor: ExplicitDiscriminantNonsplit}
	Suppose that $[Z] \in \Hilb^{2}(\bbZ)$ is nonsplit. In the notation of Lemma~\ref{Lemma: VectorInCoor}, we have
		\begin{equation}		\label{Eqn: GaloisPairDiscrim}
			\operatorname{Disc}(Z) = 4 \beta^2  g^2  \mathcal{D}/ \operatorname{GCD}^{2}(\beta^{2} \mathcal{D}, 2 \alpha \beta \mathcal{D}, g^{2}-\alpha^2 \mathcal{D})
		\end{equation}
\end{co}
\begin{proof}
	By Lemma~\ref{Lm: DescribeAlgebra}, $\operatorname{Disc}(Z)$ equals the discriminant of  $q( S e + T f)$, and we describe this polynomial in terms of $g, \alpha, \beta, \mathcal{D}$.  By construction, we have
	\begin{align*}
		(a_1 + a_2 \sqrt{\mathcal{D}}, b_1 + b_{2} \sqrt{\mathcal{D}}, c_{1} + c_{2} \sqrt{\mathcal{D}}) =& (g + \alpha \sqrt{\mathcal{D}} ) e + \beta \sqrt{\mathcal{D}} f \text{ and } \\
		(a_1 - a_2 \sqrt{\mathcal{D}}, b_1 - b_{2} \sqrt{\mathcal{D}}, c_{1} - c_{2} \sqrt{D}) =& (g - \alpha \sqrt{\mathcal{D}} ) e - \beta \sqrt{\mathcal{D}} f.
	\end{align*}
	We conclude that $q(S e + T f)$ has the same roots as
	\[
		(\beta \sqrt{D} S - (g + \alpha \sqrt{\mathcal{D}}) T) 	\cdot (-\beta \sqrt{\mathcal{D}} S - (g -\alpha \sqrt{\mathcal{D}}) T) = -\beta^{2} \mathcal{D} S^{2} + 2 \alpha \beta \mathcal{D} S T  +(g^{2}-\alpha^2 \mathcal{D}) T^{2}.
	\]
	After dividing through by the GCD of the coefficients, this last polynomial becomes primitive and thus equals $q(S e+T f)$ up to sign. Equation~\eqref{Eqn: GaloisPairDiscrim} now follows from computing the discriminant.
\end{proof}

\begin{pr} \label{Prop: DiscBound}
	If $[Z] \in \Hilb^{2}(\bbZ)$, then 
	\begin{equation} \label{Eqn: DiscBound}
		| \operatorname{Disc}(Z) | = \operatorname{Disc} H^{0}(Z, \calO_{Z}) \ll H_{-2,2}([Z]).
	\end{equation}
\end{pr}
\begin{proof}
	We will directly prove that
	\begin{equation} \label{Eqn: RewrittenDiscInequality}
		|\operatorname{Disc}(Z) | \ll \frac{H_{\Le}^2([Z])}{\operatorname{covol}^{2} I_{Z}(1)}
	\end{equation}
	for $H_{\Le}([Z])$ the height function from Definition~\ref{Def: LeRudHeight}. Inequality~\eqref{Eqn: RewrittenDiscInequality} is equivalent to the desired inequality since $H_{\Le}$ is equivalent to $H_{0,1}$ by Corollary~\ref{Cor: LeRudHtComparison}.  We handle the split and nonsplit cases separately. (There is nothing to show when $Z$ is nonreduced.)
	
	Suppose first that $Z$ is split.  Let $v = (v_{0}, v_{1}, v_{2})$ and  $w= (w_{0}, w_{1}, w_{2}) \in \bbZ^{\oplus 3}$ be two linearly independent primitive solutions to $q(x, y, z)=\ell(x, y, z)=0$.  
Then $I_{Z}(1)$ is generated by 
	\begin{equation}	
		\ell = \frac{(v \times w) \cdot (X_0, X_1, X_2)}{g} = 
		\frac{v_{1} w_{2} - v_{2} w_{1}}{g} X_0+ \frac{v_{2} w_{0} - v_{0} w_{2}}{g} X_1+ \frac{v_{0} w_{1}-v_{1} w_{0}}{g} X_2,
	\end{equation}
		where
		\[
		 g := \operatorname{GCD}(v_{1} w_{2}-v_{2} w_{1}, v_{2} w_{0} - v_{0} w_{2}, v_{0} w_{1}-v_{1} w_{0})
		 \]
		 is the greatest common divisor of the coordinates of the cross product $v \times w$.	 We thus have 
	that $\operatorname{covol} I_{Z}(1) = 
	||w|| \cdot ||v|| \cdot |\sin(\theta)|/|g|$ for $\theta$ the angle between $v$ and $w$. 
	Therefore, by  Corollary~\ref{Cor: ExplicitDiscriminantSplit} and Definition~\ref{Def: LeRudHeight}, the claim \eqref{Eqn: RewrittenDiscInequality} reduces to 
	\[
	g^2 \ll \frac{g^2}{\sin^2(\theta)},
	\]
	so that  \eqref{Eqn: RewrittenDiscInequality} holds with implied constant $1$.

	We now turn our attention to the case where  $Z$ is nonsplit.  Set $\mathcal{D}:=\operatorname{Disc}(Z)$.  Let $v \in (\bbZ[\sqrt{\mathcal{D}}])^{\oplus 3}$ be a nontrivial solution to $\ell(x, y, z)=q(x, y, z)=0$.  As in Lemma~\ref{Lemma: VectorInCoor}, we can find an integral  basis $e, f$ for the solution space of $\ell(x,y,z)=0$ and integers $g, \alpha, \beta$ with $\beta \neq 0$ such that 
	\[
		v= (g + \alpha \sqrt{\mathcal{D}}) e + (\beta \sqrt{\mathcal{D}}) f
	\]
	Another solution is then the Galois conjugate 
	\[
		\sigma(v)= (g - \alpha \sqrt{\mathcal{D}}) e - (\beta \sqrt{\mathcal{D}}) f.
	\]
	
	Corollary~\ref{Cor: ExplicitDiscriminantNonsplit} describes the left-hand side of  \eqref{Eqn: RewrittenDiscInequality}, and we compute the right-hand side as follows.  Consider the cross product:
\[
	v \times \sigma(v) = - 2 g \beta \sqrt{\mathcal{D}} \cdot e \times f.
\]
The cross product $e \times f$ is a primitive vector.  (If $u \in \bbZ^{\oplus 3}$ extends $e, f$ to a basis then $e \times f, e \times u, f \times u$ is a basis for $\bbZ^{\oplus 3}$.) A generator for $I_{Z}(1)$ is thus the linear polynomial with coefficients given by the coordinates of $e \times f$.  We conclude that 
\begin{align*}
	\operatorname{covol}^{2}( I_{Z}(1)) 	=& \frac{|| v \times \sigma(v)||^{2} }{4 g^{2}  \beta^{2}  |\mathcal{D}|}	\\
								=& \frac{|| v||^{2} \cdot ||\sigma(v)||^{2} - |\langle v, \sigma(v) \rangle|^{2} }{4 g^{2} \beta^{2}  |\mathcal{D}|}	\\
								=& \frac{|| v||^{2} \cdot ||\sigma(v)||^{2} \cdot \left(1 - \frac{|\langle v, \sigma(v) \rangle|^{2}}{ ||v||^{2} ||\sigma(v)||^{2}} \right) }{4 g^{2}  \beta^{2}  |\mathcal{D}|}.
\end{align*}
The quantity $1 - \frac{\langle v, \sigma(v) \rangle^{2}}{ ||v||^{2} ||\sigma(v)||^{2}}$ is a nonnegative number bounded by $1$ by the Cauchy--Schwarz inequality.  (It equals $\sin^{2}(\theta)$ when $v$ and $\sigma(v)$ lie in $\bbR^{\oplus 3}$.)

Since 
\[
	H_{\Le}([Z]) = ||v|| \cdot  ||\sigma(v)||/\operatorname{Norm}(I(v)),
\]
we have
\[
	\frac{H_{\Le}([Z])^{2}}{ \operatorname{covol}^{2} I_{Z}(1)} =  \frac{4  \beta^{2} g^{2}  |\mathcal{D}|}{\left(1 - \frac{|\langle v, \sigma(v) \rangle|^{2}}{ ||v||^{2} ||\sigma(v)||^{2}} \right) \operatorname{Norm}(I(v))^{2}}.
\]
Comparing this expression with \eqref{Eqn: GaloisPairDiscrim} (the expression for the discriminant in the last corollary), it is sufficient to prove 
\[
	\operatorname{Norm}(I(v)) \ll \operatorname{GCD}(\beta^{2} \mathcal{D}, 2 \alpha \beta \mathcal{D}, g^{2} -\alpha^2 \mathcal{D}).
\]

The norm $\operatorname{Norm}(I(v))$ equals the cardinality of $\calO/(a_{1} + a_{2} \sqrt{\mathcal{D}}, b_{1}+b_{2} \sqrt{\mathcal{D}}, c_{1} + c_{2} \sqrt{\mathcal{D}})$.  Here $\calO$ is the ring of integers of $\bbQ[\sqrt{\mathcal{D}}]$.  To compute the cardinality, observe first that the ideal generated by the coordinates of $v$ equals the ideal generated by  $g+\alpha \sqrt{\mathcal{D}}$ and $\beta \sqrt{\mathcal{D}}$.  The containment $(a_{1} + a_{2} \sqrt{\mathcal{D}}, b_{1}+b_{2} \sqrt{\mathcal{D}}, c_{1} + c_{2} \sqrt{\mathcal{D}}) \subset ( g+\alpha \sqrt{\mathcal{D}}, \beta \sqrt{\mathcal{D}})$ is immediate.  For the reverse inclusion, observe that, by substituting $\eqref{Eqn: VectorInCoord}$, we get that 
\begin{equation} \label{Eqn: UnknownSoln}
	g+\alpha \sqrt{\mathcal{D}} 	=  x (a_{1} + \sqrt{\mathcal{D}} a_{2}) + y (b_{1} + b_{2} \sqrt{\mathcal{D}} ) + z ( c_{1}+c_{2} \sqrt{\mathcal{D}})  \text{ with $x, y, z \in \bbZ$}
\end{equation}
is equivalent to 
\[
	g+\alpha \sqrt{\mathcal{D}} 	=  \langle (x, y, z), e \rangle \cdot (g+\alpha \sqrt{\mathcal{D}})+ \langle (x, y, z), f \rangle \cdot \beta \sqrt{\mathcal{D}}.
\]
Since $e, f$ span a primitive lattice, we can find $(x,y, z) \in \bbZ^{\oplus 3}$ such that $\langle (x, y, z), e \rangle =1$ and $\langle (x, y, z), f \rangle =0$.  These choices of $x, y, z$ solve \eqref{Eqn: UnknownSoln}, showing that $g + \alpha \sqrt{\mathcal{D}}$ lies in the ideal generated by $a_{1} + a_{2} \sqrt{\mathcal{D}}, b_{1}+b_{2} \sqrt{\mathcal{D}}, c_{1} + c_{2} \sqrt{\mathcal{D}}$.  Similarly, this ideal also contains $\beta \sqrt{\mathcal{D}}$.

We compute the cardinality of 
\[
	\frac{\calO}{( g+\alpha \sqrt{\mathcal{D}}, \beta \sqrt{\mathcal{D}})}
\]
using the theory of Smith normal form.  The theory states that, quite generally, the quotient of $\bbZ^{\oplus n}$ by the columns of a $n$-by-$m$ matrix of rank $n$ is the greatest common divisor of the $n$-by-$n$ minors.  If we express $g+\alpha \sqrt{\mathcal{D}}$ and $\beta \sqrt{\mathcal{D}}$ and their multiples by $\sqrt{\mathcal{D}}$ as column vectors using the basis $1$, $\sqrt{\mathcal{D}}$ for $\bbZ[\sqrt{\mathcal{D}}]$, then a computation of minors shows
\[
	\#\left(\bbZ[\sqrt{\mathcal{D}}]/( g+\alpha \sqrt{\mathcal{D}}, \beta \sqrt{\mathcal{D}})\right) = \operatorname{GCD}(\beta^2 \mathcal{D}, \alpha \beta \mathcal{D}, g^{2}-\alpha^2 \mathcal{D}, \beta g).
\]
We conclude that 
\[
	\#\left(\calO/( g+\alpha \sqrt{\mathcal{D}}, \beta \sqrt{\mathcal{D}})\right) \ll \operatorname{GCD}(\beta^2 \mathcal{D}, \alpha \beta \mathcal{D}, g^{2}-\alpha^2 \mathcal{D}, \beta g)
\]
since $\bbZ[\sqrt{\mathcal{D}}]=\calO$ when $\mathcal{D} \equiv 2, 3 \text{ mod $4$}$ and otherwise it is an index 2 subgroup, and this completes the proof.
\end{proof}

\begin{rmk}
	In light of Proposition~\ref{Prop: DiscBound}, it is natural to examine the quantity 
	\begin{equation} \label{Eqn: BoundRatio}
		\frac{|\operatorname{Disc}(Z)|}{ \operatorname{covol}^{2} I_{Z}(2) / \operatorname{covol}^{4} I_{Z}(1)}.
	\end{equation}
	
	The lemma shows that \eqref{Eqn: BoundRatio} is bounded above.  The quantity is trivially bounded below by zero.  A more meaningful question is: if we exclude subschemes with discriminant $0$, is \eqref{Eqn: BoundRatio} bounded below by a positive constant?  Simple examples show that no such bound exists.  Consider, for example, the subscheme defined by  $\ell = X_2$ and $q = X_0^2 - D X_1^2$ for $D \in \bbZ$ a perfect square.  A computation shows
	\begin{align*}
		\operatorname{Disc}(Z) &= 4 D, \\
		\operatorname{covol}^{2} I_{Z}(2) / \operatorname{covol}^{4} I_{Z}(1) &= D^2+1, \text{ so } \\
		\frac{|\operatorname{Disc}(Z)|}{ \operatorname{covol}^{2} I_{Z}(2) / \operatorname{covol}^{4} I_{Z}(1)}&\ll1/D.
	\end{align*}
\end{rmk}


\section*{Acknowledgments}
We would like to thank Tim Browning for several helpful comments. 

This work was partially supported by the National Science Foundation under Grant No. DMS-1201330 (F.T.), by the National Security Agency under Grants No. H98230-15-1-0264 (J.K.) and
H98230-16-1-0051 (F.T.), and by the Simons Foundation under Grants No. 429929 (J.K)., 563234 (F.T.), and 586594 (F.T.) 

The United States Government is authorized to reproduce and distribute reprints of the projects sponsored by the NSA notwithstanding any copyright notation herein.

\bibliographystyle{AJPD}
\bibliography{Kass}

\providecommand{\bysame}{\leavevmode\hbox to3em{\hrulefill}\thinspace}
\providecommand{\MR}{\relax\ifhmode\unskip\space\fi MR }
\providecommand{\MRhref}[2]{%
  \href{http://www.ams.org/mathscinet-getitem?mr=#1}{#2}
}
\providecommand{\href}[2]{#2}
\begin{thebibliography}{ABCH13}

\bibitem[ABCH13]{arcara13}
Daniele Arcara, Aaron Bertram, Izzet Coskun, and Jack Huizenga, \emph{The
  minimal model program for the {H}ilbert scheme of points on {$\Bbb{P}^2$} and
  {B}ridgeland stability}, Adv. Math. \textbf{235} (2013), 580--626.
  \MR{3010070}

\bibitem[BM90]{batyrev90}
V.~V. Batyrev and Yu.~I. Manin, \emph{Sur le nombre des points rationnels de
  hauteur born\'e des vari\'et\'es alg\'ebriques}, Math. Ann. \textbf{286}
  (1990), no.~1-3, 27--43. \MR{1032922}

\bibitem[Cas97]{Cassels}
J.~W.~S. Cassels, \emph{An introduction to the geometry of numbers}, Classics
  in Mathematics, Springer-Verlag, Berlin, 1997, Corrected reprint of the 1971
  edition. \MR{1434478}

\bibitem[Dav51]{dav_lemma}
H.~Davenport, \emph{On a principle of {L}ipschitz}, J. London Math. Soc.
  \textbf{26} (1951), 179--183. \MR{0043821}

\bibitem[Fog68]{fogarty68}
John Fogarty, \emph{Algebraic families on an algebraic surface}, Amer. J. Math
  \textbf{90} (1968), 511--521. \MR{0237496 (38 \#5778)}

\bibitem[Fog73]{fogarty73}
J.~Fogarty, \emph{Algebraic families on an algebraic surface. {II}. {T}he
  {P}icard scheme of the punctual {H}ilbert scheme}, Amer. J. Math. \textbf{95}
  (1973), 660--687. \MR{0335512 (49 \#293)}

\bibitem[Gro95]{fga}
Alexander Grothendieck, \emph{Techniques de construction et th\'{e}or\`emes
  d'existence en g\'{e}om\'{e}trie alg\'{e}brique. {IV}. {L}es sch\'{e}mas de
  {H}ilbert}, S\'{e}minaire {B}ourbaki, {V}ol. 6, Soc. Math. France, Paris,
  1995, pp.~Exp. No. 221, 249--276. \MR{1611822}

\bibitem[Har77]{hartshorne77}
Robin Hartshorne, \emph{Algebraic geometry}, Springer-Verlag, New
  York-Heidelberg, 1977, Graduate Texts in Mathematics, No. 52. \MR{0463157}

\bibitem[Har10]{hartshorne}
\bysame, \emph{Deformation theory}, Graduate Texts in Mathematics, vol. 257,
  Springer, New York, 2010. \MR{2583634}

\bibitem[HS00]{hindry00}
Marc Hindry and Joseph~H. Silverman, \emph{Diophantine geometry}, Graduate
  Texts in Mathematics, vol. 201, Springer-Verlag, New York, 2000, An
  introduction. \MR{1745599}

\bibitem[LDTT22]{LDTT}
David Lowry-Duda, Takashi Taniguchi, and Frank Thorne, \emph{Uniform bounds for
  lattice point counting and partial sums of zeta functions}, Math. Z.
  \textbf{300} (2022), no.~3, 2571--2590. \MR{4381213}

\bibitem[LQZ03]{li03}
Wei-Ping Li, Zhenbo Qin, and Qi~Zhang, \emph{Curves in the {H}ilbert schemes of
  points on surfaces}, Vector bundles and representation theory ({C}olumbia,
  {MO}, 2002), Contemp. Math., vol. 322, Amer. Math. Soc., Providence, RI,
  2003, pp.~89--96. \MR{1987741 (2004g:14004)}

\bibitem[LR14]{lerudulier}
C\'{e}cile Le~Rudulier, \emph{Points alg\'{e}briques de hauteur born\'{e}e},
  Ph.D. thesis, Universit\'{e} Rennes 1, 2014, available at
  \url{https://arxiv.org/abs/1710.02190}.

\bibitem[M{\^a}n]{manzateanu}
Adelina M{\^a}nz{\u{a}}{\c{t}}eanu, \emph{Counting points on
  $\textnormal{Hilb}^m({\P}^2)$ over function fields}, Preprint (2019),
  available at \url{https://arxiv.org/abs/1905.04772}.

\bibitem[Pey17]{peyre_freeness}
Emmanuel Peyre, \emph{Libert\'{e} et accumulation}, Doc. Math. \textbf{22}
  (2017), 1615--1659. \MR{3741845}

\bibitem[Pey21]{peyre_beyond_heights}
\bysame, \emph{Chapter {V}: {B}eyond heights: slopes and distribution of
  rational points}, Arakelov geometry and {D}iophantine applications, Lecture
  Notes in Math., vol. 2276, Springer, Cham, 2021, pp.~215--279. \MR{4238440}

\bibitem[Saw]{sawin}
Will Sawin, \emph{Freeness alone is insufficient for {M}anin-{P}eyre}, Preprint
  (2020), available at \url{https://arxiv.org/pdf/2001.06078}.

\bibitem[Sch67]{schmidt67}
Wolfgang~M. Schmidt, \emph{On heights of algebraic subspaces and diophantine
  approximations}, Ann. of Math. (2) \textbf{85} (1967), 430--472. \MR{0213301
  (35 \#4165)}

\bibitem[Sch95]{schmidt95}
\bysame, \emph{Northcott's theorem on heights. {II}. {T}he quadratic case},
  Acta Arith. \textbf{70} (1995), no.~4, 343--375. \MR{1330740}

\bibitem[Ser97]{serre97}
Jean-Pierre Serre, \emph{Lectures on the {M}ordell-{W}eil theorem}, third ed.,
  Aspects of Mathematics, Friedr. Vieweg \& Sohn, Braunschweig, 1997,
  Translated from the French and edited by Martin Brown from notes by Michel
  Waldschmidt, With a foreword by Brown and Serre. \MR{1757192 (2000m:11049)}

\end{thebibliography}

\end{document}